\documentclass[graybox]{svmult}

\usepackage{mathptmx}       
\usepackage{helvet}         
\usepackage{courier}        
\usepackage{type1cm}        

\usepackage{makeidx}         
\usepackage{graphicx}        
\usepackage{multicol}        
\usepackage[bottom]{footmisc}

\makeindex             

\usepackage{amssymb}
\usepackage{amsmath}
\usepackage[mathscr]{euscript}
\usepackage{pgf,tikz} 
\usetikzlibrary{arrows}
\usepackage[small]{caption}
\usepackage{psfrag}
\usepackage{hyperref}
\usepackage{nameref} 
\usepackage[latin1]{inputenc}
\usepackage{tcolorbox}
\usepackage{enumerate}

\usepackage{tikz}
\usetikzlibrary{arrows,decorations.pathmorphing,backgrounds,positioning,fit,petri}
\usetikzlibrary{shapes}

\newcommand{\Exp}[3]{\mathbb{E}_{#1}^{#2}\left[#3\right]}
\newcommand{\Prob}[3]{\mathbb{P}_{#1}^{#2}\left[#3\right]}
\newcommand{\mc}[1]{{\mathcal #1}}

\newcommand{\mb}[1]{{\mathbf #1}}
\newcommand{\bb}[1]{{\mathbb #1}}

\makeatletter
\let\orgdescriptionlabel\descriptionlabel
\renewcommand*{\descriptionlabel}[1]{%
	\let\orglabel\label
	\let\label\@gobble
	\phantomsection
	\edef\@currentlabel{#1}%
	\let\label\orglabel
	\orgdescriptionlabel{#1}%
}
\makeatother

\begin{document}

\title*{The boundary driven zero-range process}

\author{Susana Fr\'ometa, Ricardo Misturini and Adriana Neumann}

\institute{Susana Fr\'ometa \at UFRGS, Instituto de Matem\'atica e Estat\'istica, Campus do Vale, Av. Bento Gon\c calves, 9500. CEP 91509-900, Porto Alegre, Brasil \email{susana.frometa@ufrgs.br}
	\and Ricardo Misturini  \at UFRGS, Instituto de Matem\'atica e Estat\'istica, Campus do Vale, Av. Bento Gon\c calves, 9500. CEP 91509-900, Porto Alegre, Brasil \email{ricardo.misturini@ufrgs.br}
\and Adriana Neumann \at UFRGS, Instituto de Matem\'atica e Estat\'istica, Campus do Vale, Av. Bento Gon\c calves, 9500. CEP 91509-900, Porto Alegre, Brasil \email{aneumann@mat.ufrgs.br}}

\maketitle

\abstract*{
 We study the asymptotic behaviour of the   symmetric zero-range process   in the finite lattice $\{1,\ldots, N-1\}$ with slow boundary, in which particles are created at site $1$ or annihilated at site $N\!-\!1$ with rate proportional to $N^{-\theta}$, for $\theta\geq 1$.  
 We present the invariant measure for this model and obtain the hydrostatic limit. In order to understand the asymptotic behaviour of the spatial-temporal evolution of this model under the diffusive scaling, we start to analyze the hydrodynamic limit, exploiting attractiveness as an essential ingredient. We obtain that the hydrodynamic equation  has boundary conditions that depend on the value of $\theta$. 
}

\abstract{
	We study the asymptotic behaviour of the   symmetric zero-range process   in the finite lattice $\{1,\ldots, N-1\}$ with slow boundary, in which particles are created at site $1$ or annihilated at site $N\!-\!1$ with rate proportional to $N^{-\theta}$, for $\theta\geq 1$.  
	We present the invariant measure for this model and obtain the hydrostatic limit. In order to understand the asymptotic behaviour of the spatial-temporal evolution of this model under the diffusive scaling, we start to analyze the hydrodynamic limit, exploiting attractiveness as an essential ingredient. We obtain, through some heuristic arguments, the hydrodynamic equation, whose boundary conditions depend on $\theta$. 
}


\section{Introduction}

The zero-range process, originally introduced in $1970$ by Spitzer \cite{spitzer}, is a model that describes the behaviour of interacting particles moving on a lattice without restriction on the total number of particles per site. In this model, a particle leaves a site according to a jump rate $g(k)$ that only depends on the number of particles, $k$, in that site.  The zero-range process has been mostly studied in infinite lattices (see  \cite{andjel84,andjelvares1987,pat2,pat3,rezak1991}) and in discrete torus (see \cite{gro1,gro2,ines,bogo,kl} and the references therein). In the present work we consider the process defined in the finite lattice $I_N=\{1,\ldots,N-1\}$ with creation and annihilation of particles at the boundary. 

One of the main interest in the study of interacting particle systems is the derivation of partial differential equations (PDE) to describe the time evolution of the macroscopic density of particles as the lattice is rescaled to the continuum. Such classical scaling limit is called  \textit{hydrodynamic limit} and the associated  PDE  is called   \textit{hydrodynamic equation}. In recent years there has been an increasing interest in models that leads to hydrodynamic equations with boundary conditions (see \cite{bodineau,mariaeulalia1,mariaeulalia2,fgn1}). This has been done, for example, for the exclusion process in \cite{baldasso} and for the porous medium model in \cite{bonorino}. In both cases, the lattice $I_N$ is connected to reservoirs so that particles can be inserted into or removed from the system with rate proportional to $N^{-\theta}$, and the obtained hydrodynamic equations have boundary conditions that depend on the value of $\theta$. One common characteristic of the models in \cite{baldasso,bonorino} is that the exclusion rule only allows one particle per site, which provides a natural control for the number of particles in the system.

For the classical zero-range process in the discrete torus, see \cite[Chapter 5]{kl}, conservation of particles is an extensively used property in the proof of hydrodynamic limit, together with a hypothesis that controls the relative entropy of the initial distribution with respect to some invariant measure. In the open zero-range process, considered in the present work, the number of particles in the system is not conserved as it was in the process in the discrete torus and neither bounded as it was in the exclusion process and porous medium model. To overcome this difficulty, instead of assuming a relative entropy hypothesis, we exploit the attractiveness present in our model under the assumption that the jump rate function $g$ is non decreasing and that the initial distribution is bounded above by the invariant measure. Attractiveness was also an essential ingredient in \cite{andjelvares1987}, where  the authors obtained the hydrodynamic limit through preservation of local equilibrium for the asymmetric zero-range process on $\bb Z$ under Euler scaling. The same was done, for example, for the symmetric zero-range process in the discrete torus under the diffusive scaling, see \cite[Chapter~9]{kl}.

In this work we consider a symmetric nearest-neighbour zero-range process in $I_N$ with the following dynamics at the boundary: a particle is inserted into the system at site $1$ with rate $\alpha/N^\theta$ and removed from the system through site $N-1$ with rate $g(k)/N^\theta$, if there are $k$ particles at site $N-1$\footnote{See Remark \ref{R} for a more general dynamics allowing creation and annihilation of particles at both sides of the boundary.}, where $\alpha\geq 0$, $\theta\geq 1$ and $ g$ is same  jump rate function used in $I_N$. Computing analytically the stationary distribution of a non-equilibrium stochastic model is usually a very challenging task, see \cite{derrida19993,derrida2002,derrida2007,blythe}. However, an important general aspect of the zero-range process, that is not present in the models considered in \cite{baldasso,bonorino}, is that its invariant distribution is a product measure that can be explicitly computed, see \cite{spitzer,andjel82}. This is also true in our case, despite of the boundary conditions, as already considered in \cite{ferrari,lms2005,bertin2018}, and the resulting steady-state, when it exists, is a product measure imitating the periodic case, but now it is characterized by a non homogeneous space-dependent fugacity which is a function of the boundary rates. In Section~\ref{sectioninvariant}, we present the invariant measure for our model obtained through elementary computations involving the jump rates. Having the explicit form of the invariant measure, we obtain the stationary density profile, the so called \textit{hydrostatic limit}. 

Our main goal is to describe the asymptotic behaviour for the time evolution of the spacial density of particles for zero-range process with slow boundary introduced above. More precisely, we want to prove that, if we start our evolution with an initial configuration of particles that converges to a macroscopic density profile $\gamma:[0,1]\to\bb R_+$,  as $N\to\infty$, then, under the diffusive scaling, and in a fixed time interval $[0, T]$, the time trajectory of the spatial density of particles, $\{\pi_t^N:\;t\in[0,T]\}$, converges to a deterministic limit, $\{\pi_t:\;t\in[0,T]\}$. In the present work we prove relative compactness  for the sequence $\{\pi_t^N:\;t\in[0,T]\}$ and that the  limit points, $\{\pi_t:\;t\in[0,T]\}$, are trajectories of absolutely continuous measures on $[0,1]$, that is, $\pi_t(du)=\rho(t,u)\,du$, for  $t\in[0,T]$ and $u\in[0,1]$. We conjecture, based on some heuristic arguments, that $\rho$ is  the weak solution of the following  non-linear diffusion equation with  boundary conditions:
		\begin{equation*}
		\left\{
		\begin{array}{rcll}
		\partial_t \rho(t,u) &=& \Delta \Phi (\rho(t,u)),& \text{for } u\in(0,1)\text{ and } t\in(0,T],$$\\
		\partial_u\Phi( \rho(t,0))&=&-\kappa\,\alpha\,, & \text{for } t\in(0,T],$$\\
		\partial_u\Phi( \rho(t,1))&=&-\kappa\, \Phi( \rho(t,1))\,, & \text{for } t\in(0,T],$$\\
		\rho(0,u)&=&\gamma(u), & \text{for } u\in[0,1],
		\end{array}
		\right.
		\end{equation*}
where $\kappa=1$, if $\theta=1$, and $\kappa=0$, if $\theta>1$. The function $\Phi$ will be defined in \eqref{PHI}, in terms of  the jumps rate $g$. In Remark \ref{asymptoticprofile}, we explain what happens in the stationary regime for the case $\theta<1$.

The paper is organized as follows. In Section \ref{sectiondefinition}, we introduce some notations and define precisely the zero-range process with the boundary dynamics that we are considering. In Section \ref{sectioninvariant}, we present the invariant measure and observe the different asymptotic behaviour of the fugacity profile, depending on the value of~$\theta$. We also provide the invariant measure for a more general dynamics that allows creation and annihilation of particles in both sides of $I_N$. In Section \ref{sectionhydrostatic}, we define the notion of measures associated to a density profile and present the hydrostatic limit for our model. The small Section \ref{sectionattractiveness} is devoted to recall the essential property of attractiveness for the zero-range process. In Section \ref{tight}, we prove tightness for the sequence of probabilities of interest. For that, we introduce the related martingales that will be very useful also in the derivation of the hydrodynamic equation. In Section \ref{abscont}, we start the characterization of the limit points by showing concentration on absolutely continuous measures. In Section \ref{sectionhydrodynamic}, we present the hydrodynamic equation that we conjecture for this model, together with the necessary steps for a complete proof the of hydrodynamic limit. In Section \ref{charac_limit_points_heu}, we show how to obtain the integral form of the hydrodynamic equation from the Dynkin martingales presented in Section \ref{tight}. We use some heuristic arguments that can be formalized through some fundamental replacement lemmas, whose proof is postponed to a future work. Finally, in Section \ref{B}, we present the hydrodynamic equation obtained if we consider the general model presented in Remark \ref{R}, in which particles are created and annihilated in both sides of $I_N$.

\section{Definition of the model}\label{sectiondefinition}

Let $I_N=\{1,\ldots,N\!-\!1\}$ be the finite lattice where the distinguishable particles will be moving around, we called it by bulk. For $x\in I_N$, the occupation variable $\eta(x)$ stands for the number of particles at site~$x$. The zero-range process is an evolution without restriction on the total number of particles per site, and therefore the state space for the configurations $\eta$ is the set $\Omega_N=\bb N^{I_N}$.

\begin{figure}
	\begin{center}
		
		\begin{tikzpicture}[thick, scale=0.8][h!]
		
		
		\draw[step=1cm,gray,very thin] (-6,0) grid (6,0);
		
		\foreach \y in {-6,...,6}{
			\draw[color=black,very thin] (\y,-0.2)--(\y,0.2);
		}
		
		\draw (-6,-0.5) node [color=black] {$\scriptstyle{1}$};
		\draw (-5,-0.5) node [color=black] {$\scriptstyle{2}$};
		\draw (-4,-0.5) node [color=black] {$\scriptstyle{3}$};
		\draw (-2,-0.5) node [color=black] {$\scriptstyle{...}$};
		\draw (-1,-0.5) node [color=black] {$\scriptstyle{x-1}$};
		\draw (0,-0.5) node [color=black] {$\scriptstyle{x}$};
		\draw (1,-0.5) node [color=black] {$\scriptstyle{x+1}$};
		\draw (3,-0.5) node [color=black] {$\scriptstyle{...}$};
		\draw (5,-0.5) node [color=black] {$\scriptstyle{N\!-\!2}$};
		\draw (6,-0.5) node [color=black] {$\scriptstyle{N\!-\!1}$};

		
		\node[ball color=black!30!, shape=circle, minimum size=0.63cm] at (-6,1.2) {};
		\node[ball color=black!30!, shape=circle, minimum size=0.63cm]  at (-6.,0.4) {};
		\node[shape=circle,minimum size=0.63cm] (A) at (-6.,2) {};
		\node[ball color=black!30!,shape=circle,minimum size=0.63cm]  at (-5.,0.4) {};
		
		\node[shape=circle,minimum size=0.63cm] (C) at (-6.,1.3) {};
		\node[shape=circle,minimum size=0.63cm] (D) at (-5,1.3) {};
		\path [->] (C) edge[bend left=60] node[above right] {{\footnotesize{$g(\eta(1))$}}} (D);

		\node[shape=circle,minimum size=0.63cm] (OO) at (-7,1.3) {};
		\path [->] (OO) edge[bend left=60] node[above] {{\footnotesize{$\frac{\alpha}{N^\theta}$}}} (C);
		
%
%
%
%
		

		\node[ball color=black!30!, shape=circle, minimum size=0.63cm]  at (6.,0.4) {};
		\node[ball color=black!30!, shape=circle, minimum size=0.63cm] at (6.,1.2) {};
		\node[ball color=black!30!, shape=circle, minimum size=0.63cm] at (6.,2) {};
		\node[shape=circle,minimum size=0.63cm] (E) at (6.,2.1) {};
		\node[shape=circle,minimum size=0.63cm] (F) at (5,2.1) {};
		\path [->] (E) edge[bend right =60] node[above left] {{\footnotesize{$g(\eta(N-1))$}}} (F);
		
		\node[ball color=black!30!, shape=circle, minimum size=0.63cm]  at (5,1.2) {};
		\node[ball color=black!30!, shape=circle, minimum size=0.63cm]  at (5.,0.4) {};

		\node[shape=circle,minimum size=0.63cm] (H) at (7,2.1) {};
		\path [->] (E) edge[bend left=60] node[above] {
			{{\footnotesize{$\frac{g(\eta(N\!-\!1))}{N^\theta}$}}}}(H);
		
		\node[shape=circle,minimum size=0.63cm] (H0) at (7.1,-0.4) {};
		\node[shape=circle,minimum size=0.63cm] (E0) at (6,-0.4) {};
		
		\path [->] (H0) edge[color=white, bend left=60] node[below ] {{\footnotesize{$\frac{\beta}{N^\theta}$}}} (E0);

		\node[ball color=black!30!, shape=circle, minimum size=0.63cm] (O) at (-4,0.4) {};
		\node[ball color=black!30!, shape=circle, minimum size=0.63cm]  at (-3.,0.4) {};
		\node[ball color=black!30!, shape=circle, minimum size=0.63cm]  at (-3.,1.2) {};
		\node[ball color=black!30!, shape=circle, minimum size=0.63cm]  at (-3.,2) {};
		
		\node[ball color=black!30!, shape=circle, minimum size=0.63cm]  at (-1.,0.4) {};
		\node[ball color=black!30!, shape=circle, minimum size=0.63cm]  at (-1.,1.2) {};
		
		\node[ball color=black!30!, shape=circle, minimum size=0.63cm]  at (0.,0.4) {};
		\node[ball color=black!30!, shape=circle, minimum size=0.63cm]  at (0.,1.2) {};
		\node[ball color=black!30!, shape=circle, minimum size=0.63cm]  at (0.,2) {};
		\node[ball color=black!30!, shape=circle, minimum size=0.63cm]  at (0.,2.8) {};
		
		\node[ball color=black!30!, shape=circle, minimum size=0.63cm]  at (1.,0.4) {};
		\node[ball color=black!30!, shape=circle, minimum size=0.63cm]  at (1.,1.2) {};
		\node[ball color=black!30!, shape=circle, minimum size=0.63cm]   at (1.,2) {};

		\node[ball color=black!30!, shape=circle, minimum size=0.63cm]  at (2.,0.4) {};
		
		\node[ball color=black!30!, shape=circle, minimum size=0.63cm]  at (4,0.4) {};


		\node[shape=circle,minimum size=0.63cm] (G1) at (0,3) {};
		\node[shape=circle,minimum size=0.63cm] (H1) at (1,2.8) {};
		\path [->] (G1) edge[bend left =60] node[above] {{\footnotesize{$\quad g(\eta(x))$}}} (H1);

		\node[shape=circle,minimum size=0.63cm] (J1) at (-1,2.8) {};
		\path [<-] (J1) edge[bend left =60] node[above] {{\footnotesize{$g(\eta(x))\quad$}}} (G1);
			\end{tikzpicture}
		\caption{The boundary driven zero-range process. }\label{fig.1}
	\end{center}
\end{figure}
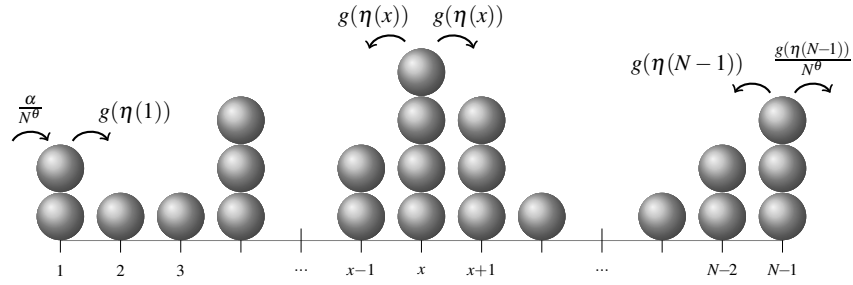

The process is defined through a function $g:\bb N \to \bb R_+$, with $g(0)=0$. 
We assume, throughout this work, that $g$ has bounded variation in the following sense: 
\begin{equation}\label{gmostlinear}
g^*=\sup_{k}|g(k+1)-g(k)|<\infty.
\end{equation}
The bulk dynamics can be described as: a particle leaves a site $x\in\{2,\dots, N\!-\! 2\}$ with rate $2g(\eta(x))$, and jumps to one of the neighbouring sites ($x-1$ or $x+1$) chosen uniformly. 
A particle jumps from the  sites $x=1$ and $x=N\!-\!1$ to a neighbour site in $I_N$ with rate $g(\eta(x))$. The boundary dynamics is given by the following birth and death processes
at the sites $x=1$ and $x=N\!-\!1$ (see Figure \ref{fig.1}).  For fixed non-negative parameters $\alpha$ and $\theta$,  a particle is  inserted into the system with rate $\alpha/N^\theta$ at site $1$  and removed  with rate $g(\eta(N-1))/N^\theta$ through the site $N\!-\!1$.\footnote {See Remark \ref{R} for a more general boundary dynamics.}    

We can entirely characterize the continuous time Markov process $\{\eta_t:\;t \geq0\}$ by its infinitesimal generator $L_N$ given by
\begin{equation}\label{generator}
L_N=L_{N,0}+L_{N,b},
\end{equation}
where $L_{N,0}$ and $L_{N,b}$ represent the infinitesimal generators of the bulk dynamics and the boundary dynamics, respectively.
The generators act on functions $f:\Omega_N\to\bb R$ as
\begin{align}
(L_{N,0}f)(\eta)=&\sum_{x=1}^{N\!-\!1}\sum_{y\in\{x-1,x+1\}\cap I_N}g(\eta(x))\;[f(\eta^{x,y})-f(\eta)],\label{jumprate1}\\
(L_{N,b}f)(\eta)=&\frac{\alpha}{N^\theta}[f(\eta^{1+})-f(\eta)]+\frac{g(\eta(N-1))}{N^\theta}[f(\eta^{(N\!-\!1)-})-f(\eta)],\label{jumprate2}
\end{align}
where $\eta^{x,y}$ represents the configuration obtained when, in the configuration $\eta$, a particle jumps from site $x$ to $y$, i.e,
\begin{equation}
\eta^{x,y}(z)=
\begin{cases}
\eta(z)\,, & \text{ if } z\neq x,y, \\
\eta(z)-1\,, & \text{ if } z=x,\\
\eta(z)+1\,, & \text{ if } z=y;
\end{cases}
\end{equation}
and $\eta^{\omega\pm}$ represents a configuration obtained from $\eta$ adding or subtracting one particle at site $\omega$, that is,
\begin{equation}
\eta^{\omega\pm}(z)=
\begin{cases}
\eta(z)\,, & \text{ if } z\neq \omega, \\
\eta(z)\pm 1\,, & \text{ if } z=\omega.
\end{cases}
\end{equation}

\begin{remark} Contrary to the classical zero-range process on the torus, see for example \cite{kl}, the process with these boundary conditions is not reversible, and does not conserve the number of particles.
\end{remark}

\section{Invariant measure}\label{sectioninvariant}

Since we do not have conservation of particles, the Markov process with generator $L_N$ is irreducible in $\Omega_N$. If the process is non-explosive and has an invariant distribution, then the invariant measure is unique and the process is positive recurrent (see ~\cite[Proposition 3.5.3]{n}). Coupling with a birth and death processes, we can see that if $g$ is such that $\sum_{k=1}^\infty\frac{1}{\max_{1\leq i \leq k}g(i)}=\infty$, then the process is non-explosive. This condition is satisfied, since we are assuming that $g$ has bounded variation, as stated in \eqref{gmostlinear}.

A particular aspect of the zero-range process is that its invariant measure can be explicitly computed (see \cite{spitzer,andjel82}). This can also be done in our case, despite of the boundary conditions, as already considered in \cite{ferrari,lms2005}. For the convenience of the reader, we will present the calculations in the following.

Inspired by the periodic case, we look for an invariant probability $\bar\nu^N$ which is a product measure on $\Omega_N$ with marginals given by 
\begin{equation}\label{invariant}
\bar\nu^N\{\eta: \eta(x)=k\}=\frac{1}{Z(\varphi(x))}\frac{(\varphi(x))^k}{g(k)!},
\end{equation}
for $x\in I_N$. Here $g(k)!$ stands for $\Pi_{1\leq j\leq k} g(j)$, and $g(0)!=1$, $\varphi:I_N\to\bb R_+$ is a function to be determined, and $Z$ is the normalizing partition function
\begin{equation}\label{partition}
Z(\varphi)=\sum_{k\geq0}\frac{\varphi^k}{g(k)!}.
\end{equation}
Denote by $\varphi^*$ the radius of convergence of the partition function \eqref{partition}.

 \begin{lemma} \label{invariantmeasurelemma}
 	For $\alpha$, $\theta$ and $N$ satisfying $\alpha(\frac{1}{N^{\theta-1}}-\frac{2}{N^\theta}+1)<\varphi^*$, the measure $\bar\nu^N$ defined in \eqref{invariant} with fugacity profile
	\begin{equation}\label{varphi}
	\varphi(x)=\varphi^N(x)=-\tfrac{\alpha}{N^\theta}(x+1)+\tfrac{\alpha}{N^{\theta-1}}+\alpha\,,\quad x\in I_N\,,
	\end{equation}
is the unique invariant distribution for the Markov process on $\Omega_N$ with infinitesimal generator $L_N$, defined in \eqref{generator}.
\end{lemma}

\begin{proof}
	Let $\eta\in\Omega_N$ be an arbitrary configuration. We have to prove that
	\begin{equation}\label{condinvar}
	\sum_{\tilde \eta \neq \eta}\frac{\bar\nu^N(\tilde \eta)}{\bar\nu^N( \eta)}R(\tilde \eta, \eta)=\lambda(\eta)\,,
	\end{equation}
	where $R(\tilde \eta, \eta)$ is the rate at which the process jumps from $\tilde \eta$ to $\eta$ and
	\begin{equation}\label{invladoA}
	\lambda(\eta)=g(\eta(1))+2\sum_{x=2}^{N-2}g(\eta(x))+g(\eta(N-1))+\frac{\alpha}{N^\theta}+\frac{ g(\eta(N-1))}{N^\theta}
	\end{equation}
	is the rate at which the process jumps from the configuration $\eta$. In the left-hand side of the equation \eqref{condinvar}, there are four types of configurations $\tilde{\eta}$ for which $R(\tilde \eta, \eta)\neq0$: $\tilde \eta=\eta^{x,x+1}$ and $\tilde \eta=\eta^{x+1,x}$, for $x \in \{1,\ldots, N-2\}$, $\tilde \eta=\eta^{1-}$ and $\tilde \eta=\eta^{(N\!-\!1)+}$. Decomposing the summation in these types of configurations, using the definition of $\bar\nu^N$ in \eqref{invariant} and the jump rates in \eqref{jumprate1} and \eqref{jumprate2}, we can rewrite the left-hand side of \eqref{condinvar} as
	\begin{equation*}
	\sum_{x=1}^{N-2}\frac{\varphi(x+1)}{\varphi(x)}g(\eta(x))+\sum_{x=1}^{N-2}\frac{\varphi(x)}{\varphi(x+1)}g(\eta(x+1))+\frac{\alpha g(\eta(1))}{N^\theta\varphi(1)}+\frac{\varphi(N\!-\!1)}{N^\theta}.
	\end{equation*}
	Thus, changing the index in the second sum above, the last expression becomes
	\begin{equation}\label{invladoB}
	\begin{split}
	&\sum_{x=2}^{N-2}\!\frac{\varphi(x\!+\!1)\!+\!\varphi(x\!-\!1)}{\varphi(x)}\,g(\eta(x))
+\frac{\varphi(2)\!+\!\tfrac{\alpha}{N^\theta}}{\varphi(1)}\,g(\eta(1))\\\,&\quad\quad\quad+\,\frac{\varphi(N\!-\!2)}{\varphi(N\!-\!1)}\,g(\eta(N-1))+\frac{\varphi(N\!-\!1)}{N^\theta}.
	\end{split}
	\end{equation}
	In order to \eqref{invladoB} be equal to \eqref{invladoA} we must require $\frac{\varphi(x+1)+\varphi(x-1)}{\varphi(x)}=2$, for all $x\in\{2, \dots, N\!-\!2\}$. To get that, choose $\varphi$ a linear function, let us say $\varphi(x)=ax+b$. The other required conditions: $\frac{\varphi(2)+\frac{\alpha}{N^\theta}}{\varphi(1)}=1$,  $\frac{\varphi(N-2)}{\varphi(N\!-\!1)}=1+\frac{1}{N^\theta}$ and $\varphi(N\!-\!1)=\alpha$, are satisfied with the choice $a=-\frac{\alpha}{N^\theta}$ and $b=\frac{\alpha}{N^\theta}(N\!-\!1)+\alpha$, which leads to \eqref{varphi}.
 	
\end{proof}

\begin{remark} The condition $\alpha(\frac{1}{N^{\theta-1}}-\frac{2}{N^\theta}+1)<\varphi^*$ imposed in Lemma \ref{invariantmeasurelemma} ensures that the fugacity function satisfies $\varphi(x)<\varphi^*$ for all $x\in I_N$. Note that if $\varphi^*$ is finite (which occurs, for instance, when $g$ is bounded), then the probability measure $\bar\nu^N$ is not well defined if $\alpha$ is too big. This is quite intuitive, since large $\alpha$ (many particles entering the system) and small $g$ (few particles leaving the system) would imply transience of the process. 
\end{remark}

A simple computation shows that $E_{\bar\nu^N}\left[g(\eta(x))\right]=\varphi^N(x)$, for $x\in I_N$, where $E_{\bar\nu^N}$ denotes expectation with respect to the measure $\bar\nu^N$. That is why $\varphi^N(x)$ is called the fugacity at the site $x$.

\begin{remark}\label{asymptoticprofile}
We observe that, depending on the value of $\theta\in[0,\infty)$, we have different asymptotic behaviours of the fugacity, see Figure \ref{fig.3}:
\begin{itemize}
	\item For $\theta=1$, for $x\in I_N$, $\varphi^N(x)=\bar\varphi(\frac{x+1}{N})$, where the asymptotic fugacity profile $\bar\varphi:[0,1]\to\bb R$ is given by $\bar\varphi(u)=\alpha(2-u)$.
	\item For $\theta>1$, $\varphi^N(x)=\alpha+r_N(x)$, where $\lim_{N\to\infty}\sup_{x\in I_N}|r_N(x)|=0$. In this case, the asymptotic fugacity profile $\bar \varphi$ is equal to the constant $\alpha$.
	\item For $\theta<1$, we must look at the two different situations: $\varphi^*<\infty$ and $\varphi^*=\infty$. If $\varphi^*<\infty$, the partition function will not be defined for large values of $N$. If $\varphi^*=\infty$, it would make sense to consider $N\to\infty$, however, we will have $\varphi^N(1)\to\infty$. Thus, as $I_N$ is rescaled to the continuum, $\varphi^N$ can not be rescaled to a macroscopic profile $\bar\varphi:[0,1]\to\bb R$, as in the previous cases.    
\end{itemize}
\end{remark}

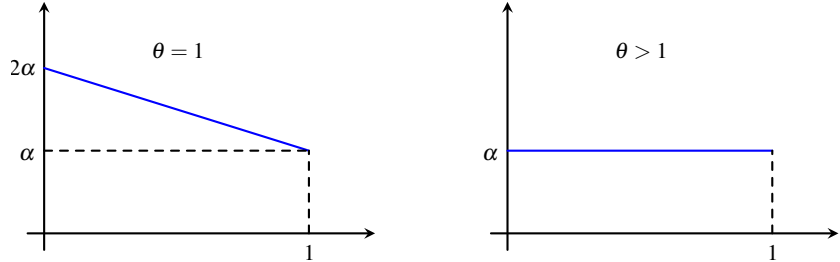
\begin{figure}
	\begin{center}
		\definecolor{qqqqff}{rgb}{0.,0.,1.}
		\begin{tikzpicture}[line cap=round,line join=round,>=stealth,x=0.44cm,y=0.44cm]
		\clip(-1,-1) rectangle (24,8);
		\draw [->,line width=0.8pt] (-0.5,0.) -- (10.,0.);
		\draw [->,line width=0.8pt] (0.,-0.5) -- (0.,7.);
		\draw [->,line width=0.8pt] (14.,-0.5) -- (14.,7.);
		\draw [->,line width=0.8pt] (13.5,0.) -- (24.,0.);
		\draw [line width=0.8pt,color=qqqqff] (0.,5.)-- (8.,2.5);
		\draw [line width=0.8pt,dash pattern=on 3pt off 3pt] (0.,2.5)-- (8.,2.5);
		\draw [line width=0.8pt,dash pattern=on 3pt off 3pt] (8.,0.)-- (8.,2.5);
		\draw (7.6,-0.12) node[anchor=north west] {$1$};
		\draw (-1,2.8) node[anchor=north west] {$\alpha$};
		\draw (-1.3,5.5) node[anchor=north west] {$2\alpha$};
		\draw [line width=0.8pt,color=qqqqff] (14.,2.5)-- (22.,2.5);
		\draw (13,2.8) node[anchor=north west] {$\alpha$};
		\draw (21.6,-0.12) node[anchor=north west] {$1$};
		\draw [line width=0.8pt,dash pattern=on 3pt off 3pt] (22.,0.)-- (22.,2.5);
		\draw (3,6) node[anchor=north west] {$\theta=1$};
		\draw (17,6) node[anchor=north west] {$\theta>1$};
		\end{tikzpicture}
		\caption{The asymptotic fugacity profile $\bar\varphi:[0,1]\to \bb R_+$ }\label{fig.3}
	\end{center}	
\end{figure}

	\begin{remark}\label{R}
	It is possible also to consider a more general model allowing creation and annihilation of particles at both sides of the boundary (see Figure \ref{fig.2}), let us say: at site $1$, particles are inserted into the system with rate $\frac{\alpha}{N^\theta}$ and removed from the system with rate $\frac{\lambda}{N^\theta}g(\eta(1))$; at site $N\!-\!1$, particles are inserted into the system with rate $\frac{\beta}{N^\theta}$ and removed from the system with rate $\frac{\delta}{N^\theta}g(\eta(N-1))$. Following the lines of Lemma \ref{invariantmeasurelemma} we found that the invariant probability is also a product measure with marginals given by \eqref{invariant} for a linear fugacity profile 

\begin{figure}
	\begin{center}
		
		\begin{tikzpicture}[thick, scale=0.8][h!]
		
		
		\draw[step=1cm,gray,very thin] (-5,0) grid (6,0);
		
		\foreach \y in {-5,...,6}{
			\draw[color=black,very thin] (\y,-0.2)--(\y,0.2);
		}
		
		\draw (-5,-0.5) node [color=black] {$\scriptstyle{1}$};
		\draw (-4,-0.5) node [color=black] {$\scriptstyle{2}$};
		\draw (-2,-0.5) node [color=black] {$\scriptstyle{...}$};
		\draw (-1,-0.5) node [color=black] {$\scriptstyle{x-1}$};
		\draw (0,-0.5) node [color=black] {$\scriptstyle{x}$};
		\draw (1,-0.5) node [color=black] {$\scriptstyle{x+1}$};
		\draw (3,-0.5) node [color=black] {$\scriptstyle{...}$};
		\draw (5,-0.5) node [color=black] {$\scriptstyle{N\!-\!2}$};
		\draw (6,-0.5) node [color=black] {$\scriptstyle{N\!-\!1}$};

		
		\node[ball color=black!30!, shape=circle, minimum size=0.63cm] at (-5,1.2) {};
		\node[ball color=black!30!, shape=circle, minimum size=0.63cm]  at (-5.,0.4) {};
		\node[shape=circle,minimum size=0.63cm] (A) at (-5.,2) {};
		\node[ball color=black!30!,shape=circle,minimum size=0.63cm]  at (-4.,0.4) {};
		
		\node[shape=circle,minimum size=0.63cm] (C) at (-6.,1.3) {};
		\node[shape=circle,minimum size=0.63cm] (OO) at (-5,1.3) {};
		\node[shape=circle,minimum size=0.63cm] (D) at (-4,1.3) {};
		\path [->] (OO) edge[bend left=60] node[above right] {{\footnotesize{$g(\eta(1))$}}} (D);

		\path [<-] (C) edge[bend left=60] node[above] {{\footnotesize{$\frac{\lambda}{N^\theta}g(\eta(1)$)}}} (OO);

		\node[shape=circle,minimum size=0.63cm] (OO0) at (-6,-0.4) {};
		\node[shape=circle,minimum size=0.63cm] (C0) at (-5,-0.4) {};
		
		\path [<-] (C0) edge[bend left=60] node[below ] {{\footnotesize{$\frac{\alpha}{N^\theta}$}}} (OO0);


		\node[ball color=black!30!, shape=circle, minimum size=0.63cm]  at (6.,0.4) {};
		\node[ball color=black!30!, shape=circle, minimum size=0.63cm] at (6.,1.2) {};
		\node[ball color=black!30!, shape=circle, minimum size=0.63cm] at (6.,2) {};
		\node[shape=circle,minimum size=0.63cm] (E) at (6.,2.1) {};
		\node[shape=circle,minimum size=0.63cm] (F) at (4.8,2.1) {};
		\path [->] (E) edge[bend right =60] node[above left] {{\footnotesize{$g(\eta(N\!-\!1))$}}} (F);
		
		\node[ball color=black!30!, shape=circle, minimum size=0.63cm]  at (5,1.2) {};
		\node[ball color=black!30!, shape=circle, minimum size=0.63cm]  at (5.,0.4) {};

		\node[shape=circle,minimum size=0.63cm] (H) at (7.2,2.1) {};
		\path [->] (E) edge[bend left=60] node[above] {
			{{\footnotesize{$\qquad\frac{\delta }{N^\theta}g(\eta(N\!-\!1))$}}}}(H);
		
		\node[shape=circle,minimum size=0.63cm] (H0) at (7.1,-0.4) {};
		\node[shape=circle,minimum size=0.63cm] (E0) at (6,-0.4) {};
		
		\path [->] (H0) edge[bend left=60] node[below ] {{\footnotesize{$\frac{\beta}{N^\theta}$}}} (E0);

		\node[ball color=black!30!, shape=circle, minimum size=0.63cm] (O) at (-4,0.4) {};
		\node[ball color=black!30!, shape=circle, minimum size=0.63cm]  at (-3.,0.4) {};
		\node[ball color=black!30!, shape=circle, minimum size=0.63cm]  at (-3.,1.2) {};
		
		\node[ball color=black!30!, shape=circle, minimum size=0.63cm]  at (-1.,0.4) {};
		\node[ball color=black!30!, shape=circle, minimum size=0.63cm]  at (-1.,1.2) {};
		
		\node[ball color=black!30!, shape=circle, minimum size=0.63cm]  at (0.,0.4) {};
		\node[ball color=black!30!, shape=circle, minimum size=0.63cm]  at (0.,1.2) {};
		\node[ball color=black!30!, shape=circle, minimum size=0.63cm]  at (0.,2) {};
		\node[ball color=black!30!, shape=circle, minimum size=0.63cm]  at (0.,2.8) {};
		
		\node[ball color=black!30!, shape=circle, minimum size=0.63cm]  at (1.,0.4) {};
		\node[ball color=black!30!, shape=circle, minimum size=0.63cm]  at (1.,1.2) {};
		\node[ball color=black!30!, shape=circle, minimum size=0.63cm]   at (1.,2) {};

		\node[ball color=black!30!, shape=circle, minimum size=0.63cm]  at (2.,0.4) {};
		
		\node[ball color=black!30!, shape=circle, minimum size=0.63cm]  at (4,0.4) {};


		\node[shape=circle,minimum size=0.63cm] (G1) at (0,3) {};
		\node[shape=circle,minimum size=0.63cm] (H1) at (1,2.8) {};
		\path [->] (G1) edge[bend left =60] node[above] {{\footnotesize{$\quad g(\eta(x))$}}} (H1);

		\node[shape=circle,minimum size=0.63cm] (J1) at (-1,2.8) {};
		\path [<-] (J1) edge[bend left =60] node[above] {{\footnotesize{$g(\eta(x))\quad$}}} (G1);
\end{tikzpicture}
		
		\caption{The general slow boundary driven zero-range process. }\label{fig.2}
	\end{center}
\end{figure}
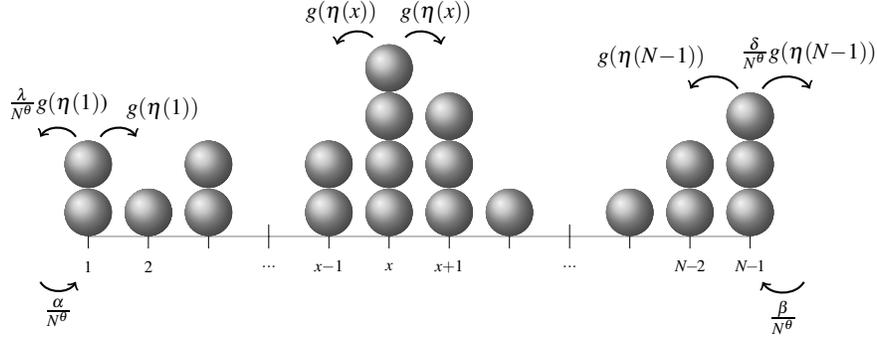	

\begin{equation*}
\varphi(x)=\varphi^N(x)=\frac{-(\alpha\delta-\beta\lambda)(x-1)+\alpha\delta(N-2)+(\alpha+\beta)N^\theta}{\lambda\delta(N-2)+(\lambda+\delta)N^\theta}\,,
\end{equation*}
for $x\in I_N$ and $\alpha,\beta,\delta, \lambda,\theta\geq 0$.
In the case $\theta=0$, this formula 
 coincides with the one presented  in \cite{lms2005}.
All results obtained in the present work, including the deduction of the hydrodynamic limit (see Section \ref{B}), can be straightforward adapted to this general case. Nevertheless, in order to avoid too much notation, we choose $\delta=1$, $\lambda=\beta=0$. Also, since the  creation of particles at one side has an analogous effect to the creation of particles at the other side, and the same holds for annihilation, the case studied in the present paper captures the essence of the macroscopic effect of the boundary dynamics. Moreover, the dynamic presented in this work has a natural interpretation as a flux of particles from a reservoir at the left-hand side of the bulk toward the one at the right-hand side. 
 \end{remark}

\section{Hydrostatic limit}\label{sectionhydrostatic}

\begin{definition}\label{associated}
	A sequence $\{\mu^N\}_{N\in \mathbb N}$ of probabilities on $\Omega_N$ is said to be \textit{associated} to the profile $\rho_0:[0,1]\to\bb R_+$ if,  for any $ \delta >0 $ and any continuous function $ H: [0,1]\to\bb R $ the following limit holds:
	\begin{equation}\label{initial_profile_integrable}
	\lim_{N\to\infty}
	\mu^{N} \Big[\, \eta \in \Omega_N:\, \Big| \frac{1}{N} \sum_{x = 1}^{N\!-\!1} H(\tfrac{x}{N})\, \eta(x)
	- \int_0^1 H(u)\, \rho_0(u)\, du \Big| > \delta\,\Big] \;=\; 0\,.
	\end{equation}
\end{definition}

Recall that $\varphi^*$ denotes the radius of convergence of the partition function $Z(\varphi)=\sum_{k\geq0}\frac{\varphi^k}{g(k)!}$. The average particle density corresponding to the fugacity $\varphi$ is a function $R:[0,\varphi^*)\to\bb R_+$, given by 
\begin{equation}\label{Rphi}
R(\varphi)=\frac{1}{Z(\varphi)}\sum_{k\geq0}k\frac{\varphi^k}{g(k)!}.
\end{equation}
As shown in \cite[Section 2.3]{kl}, $R$ is strictly increasing, and, if we assume that $$\lim_{\varphi\uparrow\varphi^*}Z(\varphi)=\infty,$$ then the range of $R$ is all $\bb R_+$, i.e, $\lim_{\varphi\uparrow\varphi^*}R(\varphi)=\infty$. Therefore, the inverse of $R$ is well defined.

\begin{equation} \label{PHI}
\mbox{Let }\Phi:\bb R_+ \to [0, \varphi^*)\mbox{ be the inverse function of }R.
\end{equation}

\begin{definition}\label{productassociated}For a continuous function $\rho_0:[0,1]\to \bb R_+$, denote by $\nu^N_{\rho_0(\cdot)}$ the product measure with slowly varying parameter associated to $\rho_0$, this is the product measure on $\Omega_N$ with marginals given by
	\begin{equation}\label{slow_varing}
	\nu^N_{\rho_0(\cdot)}\{\eta:\,\eta(x)=k\}=\frac{1}{Z(\Phi(\rho_0(\tfrac{x}{N})))}\frac{\Phi(\rho_0(\tfrac{x}{N}))^k}{g(k)!}\,,\quad\mbox{for }k\geq0\mbox{ and }x\in I_N\,.
	\end{equation}
\end{definition}
From \eqref{Rphi}, we have
\begin{equation}\label{Eeta}
E_{\nu^N_{\rho_0(\cdot)}}\left[\eta(x)\right]=\rho_0(\tfrac{x}{N})\,,\quad\mbox{for all }x\in I_N\,.
\end{equation} 

The sequence  $\{\nu^N_{\rho_0(\cdot)}\}_{N\in\bb N}$ is a particular case of a sequence of probabilities associated to the profile $\rho_0$ in the sense of Definition \ref{associated}, as stated in Proposition~\ref{measuresprofile}. To prove this, we begin with the following lemma.

\begin{lemma}\label{momentoslimitados} If $\rho_0:[0,1]\to\bb R_+$ is a continuous profile, then for each positive integer~$\ell$,
	$$\sup_{N}\sup_{x\in I_N}E_{\nu^N_{\rho_0(\cdot)}}[(\eta(x))^\ell]<\infty.$$
\end{lemma}

\begin{proof}  
	First of all, note that  $E_{\nu^N_{\rho_0(\cdot)}}[(\eta(x))^\ell]=R_\ell(\Phi(\rho_0(\tfrac{x}{N})))$, where $R_\ell$ is defined for $\varphi\in [0,\varphi^*)$ as
	$R_\ell(\varphi)=\frac{1}{Z(\varphi)}\sum_{k\geq0}k^\ell\frac{\varphi^k}{g(k)!}$. The function $\Phi=R^{-1}$ is strictly increasing and $\lim_{v\to\infty}\Phi(v)=\varphi^*$. Denote by $\varphi^{**}=\sup_{u\in[0,1]}\Phi(\rho_0(u))$. Since $\rho_0$ is bounded, we have $\varphi^{**}<\varphi^*$.
	Therefore, 
	$$\sup_{N}\sup_{x\in I_N}E_{\nu^N_{\rho_0(\cdot)}}[(\eta(x))^\ell]=\sup_{N}\sup_{x\in I_N}R_\ell(\Phi(\rho_0(x/N)))\leq\sup_{0\leq\varphi\leq\varphi^{**}}R_\ell(\varphi).$$
	In order to conclude that the last expression above is finite we observe that
	the function $R_\ell$ is analytic on $[0,\varphi^*)$. To see this, we write $R_\ell(\varphi)=\frac{A_\ell(\varphi)}{Z(\varphi)}$, where $A_\ell$ is defined inductively by $A_0(\varphi)=Z(\varphi)$ and $A_n(\varphi)=\varphi A'_{n-1}(\varphi)$.
\end{proof}

\begin{proposition}\label{measuresprofile}
	If $\rho_0:[0,1]\to\bb R_+$ is continuous, then the product measure $\nu^N_{\rho_0(\cdot)}$ defined in \eqref{slow_varing} is associated to the profile $\rho_0$ in the sense of Definition \ref{associated}.
\end{proposition}

\begin{proof} Fix a continuous test function $H$. Observing that
	$$\frac{1}{N}\sum_{x = 1}^{N\!-\!1} H(\tfrac{x}{N})\, \rho_0(\tfrac{x}{N})\to\int_0^1H(u)\rho_0(u)du,$$
	it is enough to show that, for each $\delta>0$,
	\begin{equation}\label{equivalenceofdefinitions}
	\nu^N_{\rho_0(\cdot)}\left[\eta:\,\,\left|\frac{1}{N}\sum_{x = 1}^{N\!-\!1} H(\tfrac{x}{N})[\eta(x)-\rho_0(\tfrac{x}{N})]\right|>\delta\right]
	\end{equation}
	goes to zero as $N\to\infty$. By Chebyshev's inequality, \eqref{Eeta} and independence, the expression in \eqref{equivalenceofdefinitions} is bounded above by
	$$\frac{1}{\delta^2}\frac{1}{N^2}\sum_{x=1}^{N-1}H^2(\tfrac{x}{N})E_{\nu^N_{\rho_0(\cdot)}}\left[(\eta(x)-\rho_0(\tfrac{x}{N}))^2\right]\leq \frac{1}{\delta^2}\frac{1}{N^2}\sum_{x=1}^{N-1}H^2(\tfrac{x}{N})E_{\nu^N_{\rho_0(\cdot)}}\left[\eta(x)^2\right].$$
	By Lemma \ref{momentoslimitados} and since $H$ is bounded, there exists some constant $C$ such that $H^2(\tfrac{x}{N})E_{\nu^N_{\rho_0(\cdot)}}\left[\eta(x)^2\right]<C$ for every $N$ and $x\in I_N$. Therefore, the right-hand side of the last displayed inequality goes to $0$ when $N\to\infty$.
\end{proof}

Since we have the explicit formula for the fugacity profile of the invariant measure $\bar \nu^N$, it is straightforward to obtain, in terms of the function $R$, an expression for the stationary density profile $\bar\rho:[0,1]\to\bb R_+$. Such result is usually called \textit{hydrostatic limit}. Recalling Remark \ref{asymptoticprofile}, note that, when $\theta=1$, the invariant measure $\bar\nu^N$ satisfies 
\begin{equation}
\bar\nu^N\{\eta:\eta(x)=k\}=\nu_{\bar\rho(\cdot)}^N\{\eta:\eta(x+1)=k\},
\end{equation}
where $\bar\rho(u)=R(\alpha(2-u))$. Also, when $\theta>1$, $\varphi^N(x)-\alpha$ goes to zero uniformly in $x\in I_N$, as $N\to\infty$. Therefore, the next result is derived following the lines of the proof of Proposition \ref{measuresprofile}.

\begin{proposition}[Hydrostatic Limit]\label{hydrostaticlimit}
Let $\bar\nu^N$ be the invariant measure in $\Omega_N$ for the Markov process with infinitesimal generator $L_N$. Then the sequence $\bar\nu^N$ is associated to the profile $\bar\rho:[0,1]\to \bb R_+$ given by 
\begin{equation}\label{hydrostaticprofile}
\bar\rho(u)=\begin{cases}
R(\alpha(2-u)), &\text{ if } \theta=1,\\
R(\alpha), &\text{ if } \theta>1, 
\end{cases}
\end{equation} 
for all $u\in[0,1]$.
\end{proposition}

Notice that the linear fugacity profile does not imply a linear density profile, except in the special case of non-interacting particles where $g(k)=k$.

\section{Attractiveness}\label{sectionattractiveness}
This small section is devoted to recall the essential property of attractiveness for the zero-range process. 

Consider in $\Omega_N$ the partial order: $\eta\leq \xi$ if and only if $\eta(x)\leq \xi(x)$ for every $x\in I_N$. A function $f:\Omega_N\to\bb R$ is called monotone if $f(\eta)\leq f(\xi)$ for all $\eta\leq\xi$. This partial order extends to measures on $\Omega_N$. We say that
\begin{equation}\label{defdomination}
\mu_1\leq\mu_2,  \quad\text{ if }\quad \int f d\mu_1\leq \int f d\mu_2,
\end{equation}
for all monotone functions $f:\Omega_N\to\bb R$.

An interacting particle system $\{\eta_t\}_{t\geq0}$ is said to be attractive if its semigroup $S(t)$, defined by $S(t)f(\eta)=\Exp{\eta}{}{f(\eta_t)}$, preserves the partial order: 
$$\mu_1\leq\mu_2\quad\Rightarrow\quad \mu_1S(t)\leq \mu_2S(t),$$
for all $t\geq0$. Here $\Exp{\eta}{}{f(\eta_t)}$ stands for the expectation of $f(\eta_t)$ when the process starts at $\eta(0)=\eta$.

It is well known, see \cite[Theorem 2.5.2]{kl}, that the zero-range process is attractive if $g$ is non decreasing.

\section{Tightness}\label{tight}

Let us denote by $\{\eta_t=\eta_t^N:\, t\geq0\}$ the continuous-time Markov process on $\Omega_N$ with generator $N^2L_N$. Let $\mc M_+$ be the space of positive measures on $[0,1]$ endowed with the weak topology, and denote by $\pi^N:\Omega_N\to\mc M_+$ the function that associates to each configuration $\eta$ the measure obtained by assigning mass $1/ N$ to each particle:
\begin{equation*}
\pi^N(\eta,du) =\frac{1}{N}\sum_{x=1}^{N-1}\eta(x)\delta_{\frac{x}{N}}(du),
\end{equation*}
where $\delta_u$ denotes the Dirac mass at $u$. 
Thus, the empirical process $\pi^N(\eta_t)$ is a Markov process in the space $\mc M_+$. By abuse of notation, in this section we will simply write $\pi^N_t$instead of $\pi^N(\eta_t)$.
For a function $G:[0,1]\to\bb R$, we denote by $\langle \pi^N_t, G\rangle$ the integral of $G$ with respect to the measure $\pi^N_t$:
$$\langle \pi^N_t, G\rangle=\frac{1}{N}\sum_{x\in I_N}G(\tfrac{x}{N})\eta_t(x).$$

For a measure $\mu^N$ on $\Omega^N$ we denote by $\bb P_{\mu^N}$ the probability on $\mathcal D([0,T], \Omega_N)$, the Skorohod space of c\`adl\`ag trajectories, corresponding to the jump process $\{\eta_t:\, t\geq0\}$ with generator  $N^2L_N$ and initial distribution $\mu^N$. Expectations with respect to $\bb P_{\mu^N}$ will be denoted by $\bb E_{\mu^N}$. We denote by $Q^N$ the probability on $\mathcal D([0,T], \mc M_+)$ defined by $Q^N=\bb P_{\mu^N}(\pi^N)^{-1}$.

In the next proposition we state the tightness of the sequence $\{Q^N\}_{N\geq 0}$ under the hypothesis 
\begin{equation}\label{gnondecrasing}
g(\cdot) \text{ is non decreasing},
\end{equation}
which implies attractiveness of the process.

The conservation of particles is an extensively used property in the proof of tightness for the classical zero-range process in the torus, together with a hypothesis that controls the relative entropy of the initial distribution $\mu^N$ with respect to the invariant measure; see \cite[Lemma 5.1.5]{kl}. Since we do not have conservation in our case, a different approach is necessary. Instead of a relative entropy hypothesis, we assume 
\begin{equation}\label{limitation}
\mu^N\leq\bar\nu^N,
\end{equation}
in the sense of \eqref{defdomination}, where $\bar\nu^N$ is the invariant measure. Hypothesis \eqref{limitation}, along with attractiveness, provide us a way to control the number of particles in the system, as time evolves. 

As a consequence of \cite[Lemma 2.3.5]{kl} the limitation \eqref{limitation} holds if, for instance, $\mu^N$ is a product measure of the form \eqref{invariant} associated to a fugacity function bounded above by the fugacity of the stationary measure obtained in \eqref{varphi}.

Tightness of the sequence $\{Q^N\}_{N\geq 0}$ is also true if we require that the function $g$ is bounded, instead of the hypothesis \eqref{gnondecrasing} and \eqref{limitation}. See Remark \ref{trocaatratividadeporlimitacao} for more details.

\begin{proposition}\label{proposition_tightness} Let us consider $\theta\geq 1$. Suppose that the rate function $g$ satisfies \eqref{gnondecrasing}. Assume that the sequence $\{\mu^N\}_{N\in \bb N}$ is associated to an integrable initial profile $\rho_0:[0,1]\to\bb R_+$, in the sense of \eqref{initial_profile_integrable} and satisfies~\eqref{limitation}. Then the sequence of measures $\{Q^N\}_{N\geq0}$ is tight. 
\end{proposition}

\begin{remark}\label{remark_perfil_inicial}
	Because of assumption \eqref{limitation}, the profile $\rho_0$ in the above proposition needs to be bounded above by the profile $\bar\rho$ given in \eqref{hydrostaticprofile}. A natural sequence $\{\mu^N\}_{N\in \bb N}$ satisfying the hypothesis is the sequence $\nu^N_{\rho_0(\cdot)}$ of product measures with slowly varying parameter associated to a profile $\rho_0:[0,1]\to\bb R_+$, such that $\rho_0(u)+\varepsilon\leq\bar\rho(u)$ for all $u\in[0,1]$, for some $\varepsilon>0$. 
\end{remark}

Proof of Proposition \ref{proposition_tightness} will be postponed to Subsection \ref{prooftightness}. We will introduce now the related martingales of the process studied in this work, which will be very important not only in tightness as in the whole proof of hydrodynamic limit as well.
	
\subsection{Related Martingales}\label{martingale_section}
For  $G\in C^2[0,1]$, the set of twice continuously differentiable functions in $[0,1]$, the process $M_t^G$, defined as 
\begin{equation}\label{MGt}
M_t^G=\left\langle \pi^N_t,G \right\rangle-\left\langle \pi^N_0,G \right\rangle-\int_0^tN^2L_N\left\langle \pi^N_s,G \right\rangle ds,
\end{equation}
is a martingale. Recalling the definition of the generator \eqref{generator}, we write
\begin{align}\label{N2LN}
N^2L_N\left\langle \pi^N_s,G\right\rangle  =& \frac{1}{N}\sum_{x=2}^{N-2}g(\eta_s(x))\Delta_N G\left(\tfrac{x}{N}\right) \\
&+  g(\eta_s(1))\nabla^+_N G\left(\tfrac{1}{N}\right)-  g(\eta_s(N-1))\nabla^-_NG\left(\tfrac{N-1}{N}\right)\nonumber\\
&+ \frac{\alpha}{N^{\theta-1}} G\left(\tfrac{1}{N}\right)-\frac{g(\eta_{s}(N-1))}{N^{\theta-1}}G\left(\tfrac{N-1}{N}\right)\nonumber,
\end{align}
where
\begin{align}
\Delta_NG\left(\tfrac{x}{N}\right)  &=  N^2\left[G\left(\tfrac{x+1}{N}\right)+G\left(\tfrac{x-1}{N}\right)-2G\left(\tfrac{x}{N}\right)\right],\nonumber\\
\nabla^+_NG(\tfrac{x}{N}) &= N\left[G(\tfrac{x+1}{N})-G(\tfrac{x}{N})\right]\label{discreto},\\
\nabla^-_NG(\tfrac{x}{N}) &=  N\left[G(\tfrac{x}{N})-G(\tfrac{x-1}{N})\right]\nonumber.
\end{align}

The quadratic variation of the martingale $M_t^G$ is
\begin{equation}\label{quadraticvariation}
\left\langle M^G\right\rangle_t=\int_0^t\left[N^2L_N\left\langle \pi^N_s,G\right\rangle^2-2N^2\left\langle \pi^N_s,G\right\rangle L_N\left\langle \pi^N_s,G\right\rangle\right]ds.
\end{equation}
After standard calculations we can see that $\left\langle M^G\right\rangle_t=\int_0^tB^N(s)ds$, where
\begin{align}\label{BN}
B^N(s)  = & \sum_{x=1}^{N-1}\sum_{y\in\{x-1,x+1\}\cap I_N}g(\eta_s(x))[G(\tfrac{y}{N})-G(\tfrac{x}{N})]^2 \\
& + \tfrac{\alpha}{N^\theta}G^2(\tfrac{1}{N})+\tfrac{g(\eta_{s}(N-1))}{N^\theta}G^2(\tfrac{N-1}{N}).\nonumber
\end{align}

\subsection{Proof of Tightness}\label{prooftightness}
By \cite[Proposition 4.1.7]{kl}, to prove Proposition \ref{proposition_tightness} it is sufficient to show the tightness of the measures corresponding to the real processes $\langle \pi^N_t, G\rangle$ for every $G$ in $C^2([0,1])$. By Aldous criterion, is therefore sufficient to show that the following conditions are satisfied:  

\begin{description}
	\item[Condition 1\label{condition1}] For every $t\in[0,T]$,
	$$\lim_{A\to\infty}\limsup_{N\to\infty}\Prob{\mu^N}{}{\frac{1}{N}\sum_{x=1}^{N-1}\eta_t(x)\geq A}=0.$$
	\item[Condition 2\label{condition2}] For every $\delta >0$,
	$$\lim_{\gamma\to0}\limsup_{N\to\infty}\sup_{\tau \in \mathfrak T_T\atop \omega \leq \gamma}\Prob{\mu^N}{}{\left|\frac{1}{N}\sum_{x=1}^{N-1}G(\tfrac{x}{N})\eta_{\tau+\omega}(x)-\frac{1}{N}\sum_{x=1}^{N-1}G(\tfrac{x}{N})\eta_{\tau}(x)\right|>\delta}=0,$$
	where $\mathfrak T_T$ is the family of all stopping times bounded by $T$.
\end{description}

\subsubsection*{Proof of \ref{condition1}} For $\eta\in \mc D([0,T], \Omega_N)$ define
\begin{equation}
Y_t(\eta)=\text{number of particles created up to time }t.
\end{equation}
We have the following natural bound
\begin{equation}\label{boundnumerofparticles}
\sum_{x=1}^{N-1}\eta_t(x)\leq \sum_{x=1}^{N-1}\eta_0(x)+Y_t,
\end{equation}
and then
\begin{align*}
\Prob{\mu^N}{}{\frac{1}{N}\sum_{x=1}^{N-1}\eta_t(x)\geq A}&\leq \Prob{\mu^N}{}{\frac{1}{N}\sum_{x=1}^{N-1}\eta_0(x)\geq \frac{A}{2}}+\Prob{\mu^N}{}{\frac{1}{N}Y_t\geq\frac{A}{2}}\\
&=: A_N+B_N.
\end{align*}
Not that $\lim_{A\to\infty}\limsup_{N\to\infty}A_N=0$, since $\mu^N$ is associated to and integrable profile $\rho_0$. 
On the other hand, since the process is accelerated by $N^2$, under $\bb P_{\mu^N}$,  $Y_t$ is a Poisson process with intensity $N^{2-\theta}\alpha$, and then
$$B_N\leq\frac{2}{AN}\Exp{\mu^N}{}{Y_t}=\frac{2}{AN}\cdot N^{2-\theta}\alpha t\leq\frac{2\alpha t}{A},$$
which goes to zero when $A\to\infty$.

\subsubsection*{Proof of \ref{condition2}} By \eqref{MGt}, it is enough to show that
\begin{description}
	\item[Condition 2.1\label{condition2.1}] For every $\delta>0$,
	$$\lim_{\gamma\to0}\limsup_{N\to\infty}\sup_{\tau \in \mathfrak T_T\atop \omega \leq \gamma}\Prob{\mu^N}{}{\left|\int_{\tau}^{\tau+\omega}N^2L_N\langle\pi^N_s,G\rangle ds\right|>\delta}=0.$$
	\item[Condition 2.2\label{condition2.2}] For every $\delta>0$,
	$$\lim_{\gamma\to0}\limsup_{N\to\infty}\sup_{\tau \in \mathfrak T_T\atop \omega \leq \gamma}\Prob{\mu^N}{}{\left|M_{\tau+\omega}^G-M_{\tau}^G\right|>\delta}=0.$$
\end{description}
By \eqref{N2LN}, to show \ref{condition2.1} it is sufficient to show that, for all $\delta>0$

\begin{align} \lim_{\gamma\to0}\limsup_{N\to\infty}\sup_{\tau \in \mathfrak T_T\atop \omega \leq \gamma}\Prob{\mu^N}{}{\left|\int_{\tau}^{\tau+\omega}\frac{1}{N}\sum_{x=2}^{N-2}g(\eta_s(x))\Delta_NG(\tfrac{x}{N})ds\right|>\delta} &= 0\label{c2.1.1}\\
\lim_{\gamma\to0}\limsup_{N\to\infty}\sup_{\tau \in \mathfrak T_T\atop \omega \leq \gamma}\Prob{\mu^N}{}{\left|\int_{\tau}^{\tau+\omega}g(\eta_s(1))\nabla^+_NG(\tfrac{1}{N})ds\right|>\delta} &= 0\label{c2.1.2}\\
\lim_{\gamma\to0}\limsup_{N\to\infty}\sup_{\tau \in \mathfrak T_T\atop \omega \leq \gamma}\Prob{\mu^N}{}{\left|\int_{\tau}^{\tau+\omega}g(\eta_{s}(N-1))\nabla_N^-G(\tfrac{N-1}{N})ds\right|>\delta} &= 0\label{c2.1.22}\\
\lim_{\gamma\to0}\limsup_{N\to\infty}\sup_{\tau \in \mathfrak T_T\atop \omega \leq \gamma}\Prob{\mu^N}{}{\left|\int_{\tau}^{\tau+\omega}\frac{\alpha}{N^{\theta-1}} G(\tfrac{1}{N})ds\right|>\delta} &= 0\label{c2.1.3}\\
\lim_{\gamma\to0}\limsup_{N\to\infty}\sup_{\tau \in \mathfrak T_T\atop \omega \leq \gamma}\Prob{\mu^N}{}{\left|\int_{\tau}^{\tau+\omega}\frac{g(\eta_{s}(N-1))}{N^{\theta-1}}G(\tfrac{N-1}{N})ds\right|>\delta} &= 0\label{c2.1.4}
\end{align}

Condition \eqref{c2.1.3} is immediate, since $G\in C^2([0,1])$, and then it is bounded.  

\begin{proof}[Proof of \eqref{c2.1.1}] Since $G$ is of class $C^2$ and $g$ increases at most linearly (recall hypothesis \eqref{gmostlinear}), the integral in \eqref{c2.1.1} is bounded by
	$$C(g^*, G)\int_{\tau}^{\tau+\omega}\frac{1}{N}\sum_{x=2}^{N-2}\eta_s(x)ds.$$
	By \eqref{boundnumerofparticles}, this is bounded above by
	$$C(g^*, G)\left[\frac{\omega}{N}\sum_{x=1}^{N-1}\eta_0(x)+\int_{\tau}^{\tau+\omega}\frac{1}{N}Y_s ds\right].$$
	Then, observing that $Y_s$ is non decreasing,  it is enough to show that, for any $\delta>0$
	\begin{equation}\label{c2.1.1.1}
	\lim_{\omega\to0}\limsup_{N\to\infty}\Prob{\mu^N}{}{\frac{\omega}{N}\sum_{x=1}^{N-1}\eta_0(x)> \delta}=0
	\end{equation}
	and
	\begin{equation}\label{c2.1.1.2}
	\lim_{\omega\to0}\limsup_{N\to\infty}\Prob{\mu^N}{}{\frac{\omega}{N}Y_{T+\omega}>\delta}=0.
	\end{equation}
	As in the proof of Condition 1, \eqref{c2.1.1.1} holds because $\mu^N$ is associated to an integrable profile $\rho_0$,  and \eqref{c2.1.1.2} follows from
	$$\Prob{\mu^N}{}{\frac{\omega}{N}Y_{T+\omega}>\delta}\leq\frac{\omega}{\delta N}\Exp{\mu^N}{}{Y_{T+\omega}}=\frac{\omega \alpha (T+\omega)}{\delta N^{\theta-1}}\leq\frac{\omega \alpha (T+\omega)}{\delta},$$
	which goes to zero as $\omega\to 0$.
\end{proof}

For the proof of \eqref{c2.1.2}, \eqref{c2.1.22} and \eqref{c2.1.4} we will use the following lemma.

\begin{lemma}\label{tecnico} Under the conditions \eqref{gnondecrasing} and \eqref{limitation}, for every $s\geq0$ and $x\in I_N$, it holds
	\begin{align}
	\Exp{\mu^N}{}{g(\eta_s(x))}&\leq \varphi^N(x),\label{tecnico1}\\
	\Exp{\mu^N}{}{g(\eta_s(x))^2}&\leq g^*\varphi^N(x)+(\varphi^N(x))^2.\label{tecnico2}
	\end{align}
And consequently, for $\ell=1,2$, 
\begin{equation}\label{tecnico3}
\Exp{\mu^N}{}{g(\eta_s(x))^\ell}\leq C(\alpha), 
\end{equation}
where $C(\alpha)$ is a positive constant that only depends on $\alpha$.
\end{lemma}

\begin{remark}\label{trocaatratividadeporlimitacao}
	In the proof of Proposition \ref{proposition_tightness}, the hypotheses \eqref{gnondecrasing} and \eqref{limitation} are only used in Lemma \ref{tecnico} above. Since this result is trivial when $g$ is bounded, in this case such hypotheses are not needed to prove tightness. 
\end{remark}
\begin{proof} 
	For every $x\in I_N$, by \eqref{gnondecrasing} the function $h_x:\Omega_N\to\bb R$, given by $h_x(\eta)=[g(\eta(x))]^\ell$, is monotone. So, by attractiveness and hypothesis \eqref{limitation}, we have $$\Exp{\mu^N}{}{g(\eta_s(x))^\ell}\leq \Exp{\bar\nu^N}{}{g(\eta_s(x))^\ell}=\Exp{\bar\nu^N}{}{g(\eta_0(x))^\ell}=E_{\bar\nu^N}\left[g(\eta(x))^\ell\right].$$
	To conclude the proof of \eqref{tecnico1}, we recall that $E_{\bar\nu^N}\left[g(\eta(x))\right]=\varphi^N(x)$.

	For the proof of \eqref{tecnico2}, we write
	\begin{align*}
	E_{\bar\nu^N}\left[g(\eta(x))^2\right]  &= \frac{1}{Z(\varphi^N(x))}\sum_{k=0}^\infty g(k)^2\frac{\varphi^N(x)^k}{g(k)!}\\
	&= \frac{\varphi^N(x)}{Z(\varphi^N(x))}\sum_{k=1}^\infty g(k)\frac{\varphi^N(x)^{k-1}}{g(k-1)!}.
	\end{align*}
	
	By \eqref{gmostlinear}, we have that
	$$g(k)\leq g^* +g(k-1).$$
	Then
	\begin{align*}
	E_{\bar\nu^N}\left[g(\eta(x))^2\right]  & \leq   g^*\varphi^N(x)+ \frac{\varphi^N(x)}{Z(\varphi^N(x))}\sum_{k=1}^\infty g(k-1)\frac{\varphi^N(x)^{k-1}}{g(k-1)!}\\
	& = g^*\varphi^N(x) + \frac{\varphi^N(x)^2}{Z(\varphi^N(x))}\sum_{k=2}^\infty \frac{\varphi^N(x)^{k-2}}{g(k-2)!}\\
	& =  g^*\varphi^N(x)+ \varphi^N(x)^2.
	\end{align*}
	Since $\varphi^N$ is a linear function satisfying $\varphi^N(N-1)=\alpha$ and, for every $\theta\geq 1$, $\varphi^N(1)\leq 2\alpha$, the proof Lemma \ref{tecnico} is concluded.
\end{proof}

\begin{proof}[Proof of \eqref{c2.1.2}, \eqref{c2.1.22} and \eqref{c2.1.4}] Since $G$ is of class $C^2 $, the integrals in	\eqref{c2.1.2}, \eqref{c2.1.22} and \eqref{c2.1.4} are bounded above by
	$$C(g^*,G)\int_{\tau}^{\tau+\omega}g(\eta_s(x))ds,$$
	for $x=1$ or $x=N-1$.
	
	For all $x\in I_N$, we have
	$$\Prob{\mu^N}{}{\int_{\tau}^{\tau+\omega}g(\eta_s(x))ds>\delta}\leq \frac{1}{\delta}\Exp{\mu^N}{}{\int_{\tau}^{\tau+\omega}g(\eta_s(x))ds}.$$
	By Cauchy-Schwarz's inequality
	\begin{align*}
	\Exp{\mu^N}{}{\int_{\tau}^{\tau+\omega}g(\eta_s(x))ds} & =  \Exp{\mu^N}{}{\int_0^T \mb{1}_{[\tau,\tau+\omega]}(s)g(\eta_s(x))ds}\\
	& \leq  \sqrt{\omega} \left[\Exp{\mu^N}{}{\int_0^T g(\eta_s(x))^2ds}\right]^{1/2}\\
	& = \sqrt{\omega} \left[\int_0^T \Exp{\mu^N}{}{g(\eta_s(x))^2}ds\right]^{1/2}.
	\end{align*}
	Then, using Lemma \ref{tecnico}, we obtain 
	\begin{equation}\label{Cauchy-Schwarz's}
	\Exp{\mu^N}{}{\int_{\tau}^{\tau+\omega}g(\eta_s(x))ds}\leq (\omega TC(\alpha))^{1/2}.
	\end{equation}
	Sending $\omega\to0$, we conclude the proof.
\end{proof}

\begin{proof}[Proof of \ref{condition2.2}] Using Chebychev's inequality and the explicit formula for the quadratic variation given in \eqref{quadraticvariation}, we have
	\begin{align}
	\Prob{\mu^N}{}{\left|M_{\tau+\omega}^G-M_{\tau}^G\right|>\delta} & \leq   \frac{1}{\delta^2}\Exp{\mu^N}{}{(M_{\tau+\omega}^G-M_{\tau}^G)^2}\nonumber\\
	& =  \frac{1}{\delta^2}\Exp{\mu^N}{}{\int_\tau^{\tau+\omega}B^N(s)ds},\label{BN(s)bound}
	\end{align}
	where $B^N(s)$ was defined in \eqref{BN}. 
	
	Using that $G$ and its derivative are bounded functions, and then, using \eqref{Cauchy-Schwarz's}, we can see that  \eqref{BN(s)bound} is bounded above by $\frac{C\omega}{N}$, where $C$ is a constant that does not depends on $N$ and $\omega$. Thus the proof is concluded.
\end{proof}

\begin{remark}\label{possiblegeneralization}
	Considering a model in which a particle is removed from the system throughout site $N-1$ with rate $g(\eta(N-1))$ instead of the slow boundary assumption $\frac{g(\eta(N-1))}{N^\theta}$ made in this work, our proof can be adapted and tightness will also hold if we assume that particles are inserted into the system at site $1$ with rate $\frac{\alpha}{N^\theta}$ with $\theta>1$. In this case the fugacity profile $\varphi^N$ goes to zero uniformly as $N\to\infty$.
\end{remark}

\section{Limit points are concentrated on absolutely continuous  measures}\label{abscont}

The next step to characterize the limit points of $\{Q^N\}$ is to show that they are concentrated on trajectories of measures that are absolutely continuous with respect to the Lebesgue measure.

Next lemma states that for any sequence $\mu^N$ of probabilities on $\Omega_N$ bounded by the invariant measure $\bar\nu^N$, the corresponding sequence of empirical measures, obtained via $\pi^N:\Omega^N\to\mc M_+$, if converges, must converge to an absolutely continuous measure with respect to Lesbegue.

\begin{lemma}\label{lema_absolute_continuity} Let $\mu^N$ be a sequence of probabilities on $\Omega_N$ bounded by the invariant measure $\bar\nu^N$, i.e., $\mu^N\leq \bar\nu^N$. Let $R_{\mu^N}$ be the probability measure $\mu^N(\pi^N)^{-1}$ on $\mc M_+$, defined by
$R_{\mu^N}(\mc A)=\mu^N\{\eta: \pi^N(\eta)\in \mc A\}$
for every Borel subset $\mc A\in \mc M_+$. Then, all limit points $R^*$ of the sequence $R_{\mu^N}$ are concentrated on absolutely continuous measures with respect to the Lebesgue measure:
	$$R^*[\pi: \pi(du)=\rho(u)du]=1.$$
\end{lemma}

\begin{proof} Let $R^*$ be a limit point of the sequence $R_{\mu^N}$. Recall from \eqref{hydrostaticprofile} that we denoted by $\bar\rho:[0,1]\to\bb R_+$ the density profile associated to the sequence of invariant measures $\bar\nu^N$. Fix some $\varepsilon>0$, it is enough to prove that, for every non negative continuous function $G:[0,1]\to \bb R$,
	$$R^*\left[\pi:\,\langle \pi, G \rangle \leq \int_0^1 G(u)(\bar\rho(u)+\varepsilon)du \right]=1.$$
	Let $R_{\mu^{N_k}}$ be a subsequence converging to $R^*$, then
	
	\begin{align}
	R^*&\left[\pi:\,\langle \pi, G \rangle \leq \int_0^1 G(u)(\bar\rho(u)+\varepsilon)du \right]\nonumber\\
	&\geq \limsup_{k\to\infty}R_{\mu^{N_k}}\left[\pi:\,\langle \pi, G \rangle \leq \int_0^1 G(u)(\bar\rho(u)+\varepsilon)du\right] \label{abso1} \\
	&= \limsup_{k\to\infty} \mu^{N_k}\left[\eta:\langle \pi^N(\eta),G\rangle \leq \int_0^1 G(u)(\bar\rho(u)+\varepsilon)du\right].\nonumber
	\end{align}
	Since $\mu^N\leq \bar\nu^N$, by \cite[Theorem 2.2.4]{l} there exist a coupling $\bar\mu^N$, i.e, a probability measure on $\Omega_N\times \Omega_N$, with marginals $\mu^N$ and $\bar\nu^N$ respectively, such that
	$$\bar\mu^N\left[(\eta, \xi):\,\eta\leq\xi\right]=1,$$ 
	and consequently
	\begin{equation}\label{abso2}
	\bar\mu^N\left[(\eta,\xi):\,\langle \pi^N(\eta),G\rangle\leq \langle \pi^N(\xi),G\rangle\right]=1.
	\end{equation}
	By \eqref{abso1} and \eqref{abso2},
	\begin{align}
	R^*&\left[\pi:\,\langle \pi, G \rangle \leq \int_0^1 G(u)(\bar\rho(u)+\varepsilon)du\right]\nonumber\\
	&\geq \limsup_{k\to\infty}\bar \nu^{N_k}\left[\eta:\langle \pi^N(\eta),G\rangle \leq \int_0^1 G(u)(\bar\rho(u)+\varepsilon)du\right]=1,\label{abso3}
	\end{align}
by Proposition \ref{hydrostaticlimit}.
\end{proof}

Assuming that the rate function $g$ is non decreasing, by attractiveness, the semigroup $S^N(t)$ associated to the generator $N^2L_N$ preserves the partial order $\mu^N\leq\bar\nu^N$, that is, $\mu^NS^N(t)\leq \bar\nu^NS^N(t)=\bar\nu^N$ for each $0\leq t \leq T$. Therefore, Lemma \ref{lema_absolute_continuity}, when applied to the marginal at time $t$ of the measure $Q^N=\bb P_{\mu^N}(\pi^N)^{-1}$, which is $\mu^NS^N(t)$, says that for every limit point $Q^*$, and every $t\in[0,T]$,
$$Q^*\left[\pi: \pi_t(du)=\rho_t(u)du\right]=1.$$
To short notations, we write $\rho_{t}(u)$ instead of  $\rho(t,u)$.
Now consider the functional $J:\mc M^+\to \bb R_+\cup\{\infty\}$ defined by
$$J(\pi)=\begin{cases} 1, & \text{ if } \pi(du)=\rho(u)du,\\
\infty, & \text{otherwise.}
\end{cases}$$
By Fubini's lemma, 
$$E_{Q^*}\left[\int_0^TJ(\pi_t)dt\right]=\int_0^T E_{Q^*}\left[J(\pi_t)\right]dt=T.$$
In particular, changing, if necessary, $\pi_t(du)$ in a time set of measure zero, all limit points $Q^*$ are concentrated on absolutely continuous trajectories:
$$Q^*\left[\pi_{\cdot}: \; \pi_t(du)=\rho_t(u)du, ~~ 0\leq t \leq T\right]=1.$$

\section{Hydrodynamic limit}\label{sectionhydrodynamic}

In this section we will present the hydrodynamic limit that we expect in this model, together with the structure of the proof. Since some elements of the proof are not yet completed, we present it as a conjecture.

Let us recall the hypotheses assumed in Sections \ref{tight} and \ref{abscont}, that is, $\theta\geq 1$ and $g$ is a non decreasing function with bounded variation, as stated in \eqref{gmostlinear}.
 Also recall that, for $T>0$, $\bb P_{\mu^N}$  denotes the probability on the space $\mc D([0,T], \Omega_N)$ corresponding to the process $\{\eta_t:\;t\in[0,T]\}$ on $\Omega_N$ with infinitesimal generator $N^2L_N$, where $L_N$ is defined in \eqref{generator}.

\begin{conjecture}[Hydrodynamic limit]\label{maintheorem} Let $\lbrace\mu^N\rbrace_{N \in \mathbb{N}}$ be a sequence of probability measures on $\Omega_N$, bounded by the invariant measure, i.e., $\mu^N\leq\bar\nu^N$. Assume that the sequence $\lbrace\mu^N\rbrace_{N \in \mathbb{N}}$ is associated to a continuous profile\footnote{As discussed in Remark \ref{remark_perfil_inicial}. the assumption \eqref{limitation} naturally imposes the initial profile $\gamma$ to be bounded above by the profile $\bar\rho$ of the hydrostatic limit, given in \eqref{hydrostaticprofile}.} $\gamma:[0,1]\to\bb R_+$ in the sense of the Definition \ref{associated}.
Then,  for all $t \in [0,T]$, for all continuous function $G:[0,1]\to\mathbb R$ and  $\delta>0$,
\begin{equation*}
	\lim_{N \to +\infty}\mathbb{P}_{\mu^N}\left[\eta_\cdot: 
	\,\Big| \tfrac{1}{N}\sum_{x\in I_N}G\left(\tfrac{x}{N}\right)\eta_t(x) - \int_{0}^1 G(u)\, \rho_t(u)\, du\  \,  \Big| > \delta \right]=0,
	\end{equation*}
	where
	\begin{itemize}
		\item for $\theta =1$,  $\rho_{t}(u)$ is a weak solution of \eqref{eq_forte} with Robin boundary condition ($\kappa=1$);
		\item for $\theta >1$,  $\rho_{t}(u)$ is a weak solution of \eqref{eq_forte} with Neumann boundary condition ($\kappa=0$).
	\end{itemize}
\end{conjecture}

Before introducing the hydrodynamic equation \eqref{eq_forte}, we need to define some function spaces.
 The bracket $ \langle \,\cdot\, ,\, \cdot\,\rangle$ means the inner product in $L^2[0,1]$ and $\Vert F\Vert_2^2=\langle F ,F\rangle$, for all $F\in L^2[0,1]$. 
 
 We advertise that, to short the  notation, we write $\rho_{t}(u)$ and  $G_s(u)$ instead of $\rho(t,u)$ and $G(s,u)$, respectively. 
 The reader must not misunderstand this notation with the time derivative, denoted by $\partial_s$.
\begin{definition}
	\label{Def. Sobolev space}
	Let $\mathcal{H}^{1}(0,1)$ be the set of all locally summable functions $\xi: [0,1] \to \mathbb{R}$ such that there exists a function $\partial_{u} \xi\in L^{2}[0,1]$ satisfying
	$$\langle \partial_u G, \xi \rangle =-\langle G, \partial_{u} \xi\rangle,$$
	for all $C^\infty$ function $G:(0,1) \to\mathbb R$ with compact support. 

	Let $L^2(0,T; \mathcal{H}^1(0,1))$ be the set of all measurable functions $\bar\xi: [0,T]\to L^2[0,1]$ such that $\bar\xi_t\in\mathcal{H}^1(0,1)$, for almost $t\in[0,T]$, and 
	\begin{equation}\label{sobolev norm 2}
	\|\bar \xi\|^{2}_{L^2(0,T;\mathcal{H}^1(0,1))} := \int_{0}^{T}\{\|\bar\xi_t\|^{2}_{2} + \|\partial_{u} \bar\xi_t\|^{2}_{2}\}\, dt < \infty.
	\end{equation}
\end{definition}
Denote by
$C^{1,2}([0, T] \times [0,1])$ the set of real-valued functions defined on $[0,T] \times [0,1] $ that are differentiable on the first variable and twice differentiable on the second variable.

 Recall that the function $\Phi:\bb R_+ \to [0, \varphi^*)$ is the inverse function of $R$, defined in \eqref{Rphi}.

\begin{definition}[Hydrodynamic equation] 
	Let $\gamma:[0,1]\to \bb R_+$ be a continuous function. Consider the parameter $\kappa$ equal to $0$ or $1$. We say that a function $\rho:[0,T]\times[0,1]\to\bb R_+$ is a weak solution of the equation
	\begin{equation}\label{eq_forte}
	\left\{
	\begin{array}{rcll}
	\partial_t \rho_t(u) &=& \Delta \Phi (\rho_t(u)),& \text{for } u\in(0,1)\text{ and } t\in(0,T],$$\\
	\partial_u\Phi( \rho_t(0))&=&-\kappa\,\alpha, & \text{for } t\in(0,T],$$\\
	\partial_u\Phi( \rho_t(1))&=&-\kappa\, \Phi( \rho_t(1)), & \text{for } t\in(0,T],$$\\
	\rho_0(u)&=&\gamma(u), & \text{for } u\in[0,1],
	\end{array}
	\right.
	\end{equation}
	if $\Phi(\rho) \in L^2(0,T;\mathcal H^1(0,1))$  and
	\begin{align}\label{eq_int}
	\langle \rho_{t}, G_{0}\rangle-\langle \gamma , G_{0}\rangle  &-\;  \int_{0}^{t}\big\{ \langle \rho_{s} , \partial_{s}  G_{s} \rangle + \langle \Phi(\rho_s) , \Delta G_s \rangle\big\}  \,ds\nonumber\\  
	&-\int_{0}^{t}\big\{\Phi(\rho_s(0))\partial_u G_s(0)-\Phi(\rho_s(1))\partial_u G_s(1)\big\} \,ds\nonumber \\
		& -\,\kappa \int_{0}^{t}\big\{\alpha G_s(0)-\Phi(\rho_s(1))G_s(1)\big\} \,ds\,=\,0 \,,
	\end{align}
	for all $t\in[0,T]$ and $G\in C^{1,2}([0,T]\times[0,1])$. 
	\begin{itemize}
		\item When $\kappa=0$, we say that the PDE \eqref{eq_forte} has Neumann boundary condition;
		\item When $\kappa=1$, we say that the PDE \eqref{eq_forte} has Robin boundary condition.
	\end{itemize}
\end{definition}
We consider the PDE \eqref{eq_forte} with
more general boundary conditions in the Section \ref{B}, see equation \eqref{general_hid_eq}.

Let us recall, from the beginning of Section \ref{tight}, that $Q^N$ denotes the probability on $\mc D([0,T],\mc M_+)$, corresponding to the empirical process $\{\pi_t^N:\;t\in[0,T]\}$.  

\begin{conjecture}\label{pro:conv}
	As $N\to\infty$, the sequence of probabilities $ \{ Q^N\}_{N \in \mathbb N} $ converges weakly to $Q$, the probability measure on $\mathcal D([0,T], \mc M_+)$ that gives mass one to the trajectory $ \pi_t(du) = \rho_t(u)du $, where $ \rho:[0,T]\times[0,1]\to\bb R$ is the weak solution of the hydrodynamic equation \eqref{eq_forte}, with $\kappa=1$  if $\theta=1$, and $\kappa=0$ if $\theta>1$. We call $\rho_t(u)$ the hydrodynamic profile.
\end{conjecture}

Conjecture \ref{maintheorem} is a consequence of the Conjecture 2.
		Since the hydrodynamic behaviour is described by Conjecture \ref{maintheorem}, it is not mandatory to get Conjecture 2 to understand the hydrodynamics.  But the second conjecture gives more information about the asymptotic behaviour of the system, because it handles with the whole time-trajectory of the density of particles. Moreover, using   the result stated in Conjecture 2 is possible to study the fluctuations  and the large deviations of this convergence, completing the asymptotic characterization of  the model.

The proof of Conjecture \ref{pro:conv} may be divided into three steps. 

The first step is to show tightness, which is done in Section \ref{tight}. This implies that the sequence  $ \{ Q^N\}_{N \in \mathbb N} $ has limit points. 

The second step is the characterization of these limit points, which we split in two parts: The first part is the subject of Section \ref{abscont}, where we proved that the limit points of the sequence $\{Q^N\}$ are  concentrated on trajectories of measures that are absolutely continuous with respect to the Lebesgue measure, so that for each $t$, $\pi_t(du)=\rho_t(u)du$ for some function $\rho:[0,T]\times[0,1]\to\bb R_+$.  The second part is to show that $\rho$ is a solution of the corresponding hydrodynamic equation. This part we postpone to a future work, however, in the Section  \ref{charac_limit_points_heu}, we present some heuristics of this proof. 

The third step, which will be also postponed to a future work, is to show the uniqueness of solution for the hydrodynamic equations. This uniqueness would guarantee that the sequence $\{Q^N\}$ has a unique limit point, and thus the proof of Conjecture  \ref{pro:conv} would be concluded.

%

\section{Heuristics of the hydrodynamic equation}\label{charac_limit_points_heu}

Let $\lbrace\mu^N\rbrace_{N \in \mathbb{N}}$ be a sequence of probability measures on $\Omega_N$, bounded by the invariant measure, as stated in \eqref{limitation}, and associated with to a continuous profile $\gamma:[0,1]\to\bb R_+$ in the sense of Definition \ref{associated}.
Recall that $\{Q^N\}$ is a sequence of probabilities on $\mathcal D([0,T], \mc M_+)$ defined by $Q^N=\bb P_{\mu^N}(\pi^N)^{-1}$.

Let $ Q^* $ be a limit point of $\{Q^N\}$. In Section \ref{abscont}, we proved that  $Q^*$ 
is a probability measure on $\mathcal D([0,T], \mc M_+)$  which gives mass one to paths of absolutely continuous measures: $ \pi_t(du) = \rho_t(u)du $. In this section we present some heuristics to obtain that  $\rho_{t}(u)$ is a weak solution of the corresponding hydrodynamic equation. For this purpose, we will assume, \emph{heuristically}, that 
\begin{equation}\label{assume}
\left\langle \pi^N_s,H \right\rangle\to \int_0^1H(u)\,\rho_s(u)\,du\,,
\end{equation}
when $N\to \infty$, for all $s\in [0,T]$ and $H\in C[0,1]$. 

In order to prove that $\rho_{t}(u)$ satisfies the hydrodynamic equation, we evoke the Dynkin martingale, introduced in Subsection \ref{martingale_section}, $M_t^G$ for  $G\in C^2([0,1])$. Using  \eqref{MGt} and \eqref{N2LN},   we rewrite this martingale as 
\begin{align}\label{int_0}
M_t^G= & \left\langle \pi^N_t,G \right\rangle-\left\langle \pi^N_0,G \right\rangle-\int_0^t \frac{1}{N}\sum_{x=2}^{N-2}g(\eta_s(x))\Delta_N G\left(\tfrac{x}{N}\right)\,ds\nonumber \\
&- \int_0^t \Big(g(\eta_s(1))\nabla^+_N G\left(\tfrac{1}{N}\right)-  g(\eta_{s}(N-1))\nabla^-_NG\left(\tfrac{N-1}{N}\right)\Big)\,ds\\
&- \int_0^t\Big(\frac{\alpha}{N^{\theta-1}}\, G\left(\tfrac{1}{N}\right)-\frac{g(\eta_{s}(N-1))}{N^{\theta-1}}\,G\left(\tfrac{N-1}{N}\right)\Big)\,ds.\nonumber
\end{align}

By the definition of $\Delta_NG$, $\nabla^+_NG$ and $\nabla^-_NG$ in \eqref{discreto}
and the fact that $G\in C^2([0,1])$, we can rewrite  $M_t^G$ as 
\begin{align}\label{mart_0}
&\left\langle \pi^N_t,G \right\rangle-\left\langle \pi^N_0,G \right\rangle-\int_0^t \frac{1}{N}\sum_{x=2}^{N-2}g(\eta_s(x))\Delta G\left(\tfrac{x}{N}\right)\,ds\nonumber \\
&- \int_0^t \Big(g(\eta_s(1))\partial_u G\left(0\right)-  g(\eta_{s}(N-1))\partial_u G\left(1\right)\Big)\,ds\\
&- \int_0^t\Big(\frac{\alpha}{N^{\theta-1}}\, G\left(0\right)-\frac{g(\eta_{s}(N-1))}{N^{\theta-1}}\,G\left(1\right)\Big)\,ds\,+\, \mathcal R^{1,\theta}_N(G,t)\,,\nonumber
\end{align}
where $\mathbb E_{\mu^N}[\mathcal R^{1,\theta}_N(G,t)]$ goes to zero, as $N\to \infty$, for all $\theta\geq 1$, and uniformly on $t\in[0,T]$,  because of Lemma \ref{tecnico} and Taylor's expansion. By \eqref{assume}, as $N\to\infty$,
$$\left\langle \pi^N_t,G \right\rangle-\left\langle \pi^N_0,G \right\rangle\to\left\langle \rho_t,G \right\rangle-\left\langle \rho_0,G \right\rangle.$$

Then we need to study the bulk and boundary terms of the expression \ref{mart_0}.
We start by the bulk term: $\int_0^t \frac{1}{N}\sum_{x=2}^{N-2}g(\eta_s(x))\Delta G\left(\tfrac{x}{N}\right)\,ds$. In order to do this, we will introduce some notation.

For $\varepsilon>0$, consider the set
\begin{equation}\label{available replac}
I^{\varepsilon }_{N}:=\{1+\varepsilon N, \ldots, N-1-\varepsilon N\}\,.
\end{equation} 
 Above and in all text $\varepsilon N$ must be understood as $\lfloor\varepsilon N\rfloor$.

Then the bulk term becomes 
\begin{equation}\label{int_1}\begin{split}
&\int_0^t \frac{1}{N}\sum_{x\in I^{\varepsilon }_{N}}g(\eta_s(x))\Delta G\left(\tfrac{x}{N}\right)\,ds +\, \mathcal R^{2}_{N,\varepsilon}(G,t)\,,
\end{split}
\end{equation}
where 
$$\mathcal R^{2}_{N,\varepsilon}(G,t)=\int_0^t \frac{1}{N}\left(\sum_{x=2}^{\varepsilon N}g(\eta_s(x))\Delta G\left(\tfrac{x}{N}\right)+\sum_{x=N-\varepsilon N}^{N-2}g(\eta_s(x))\Delta G\left(\tfrac{x}{N}\right)\right)ds \,.$$
Using \eqref{tecnico3} from Lemma \ref{tecnico}, we have $\mathbb E_{\mu^N}[\mathcal R^{2}_{N,\varepsilon}(G,t)]\leq 2\varepsilon CT\Vert\Delta G\Vert_\infty$.
Thus, $\mathbb E_{\mu^N}[\mathcal R^{2}_{N,\varepsilon}(G,t)]$ goes to zero, when $N\to\infty$ and  $\varepsilon\to 0$.
Adding and subtracting suitable terms, we can see that the integral in \eqref{int_1} is equal to 
\begin{equation}\label{int_2}\begin{split}
&\int_0^t \frac{1}{N}\sum_{x\in I^{\varepsilon }_{N}}\Delta G\left(\tfrac{x}{N}\right)\frac{1}{\varepsilon N}\sum_{y=x+1}^{x+\varepsilon N}g(\eta_s(y))\,ds +\, \mathcal R^{3}_{N,\varepsilon}(G,t)\,,
\end{split}
\end{equation}
where
\begin{equation*}
\begin{split}
\mathcal R^{3}_{N,\varepsilon}(G,t)=&\int_0^t \frac{1}{N}\sum_{x\in I^{\varepsilon }_{N}}\Big\{g(\eta_s(x))-\frac{1}{\varepsilon N}\sum_{y=x+1}^{x+\varepsilon N}g(\eta_s(y))\Big\}\Delta G\left(\tfrac{x}{N}\right)\,ds\,.
\end{split}
\end{equation*}
 
Changing variables $\mathcal R^{3}_{N,\varepsilon}(G,t)$ can be rewritten as 
\begin{equation*}
\begin{split}
&\int_0^t \frac{1}{N}\Bigg\{\sum_{x\in I^{\varepsilon }_{N}}g(\eta_s(x))\Delta G\left(\tfrac{x}{N}\right)-\sum_{y\in I_{N}}g(\eta_s(y))\frac{1}{\varepsilon N}\sum_{z=1}^{\varepsilon N}\Delta G\left(\tfrac{y-z}{N}\right)\Bigg\}\,ds\\
=&\int_0^t \frac{1}{N}\sum_{x\in I^{\varepsilon }_{N}}g(\eta_s(x))\Big\{\Delta G\left(\tfrac{x}{N}\right)-\frac{1}{\varepsilon N}\sum_{z=1}^{\varepsilon N}\Delta G\left(\tfrac{x-z}{N}\right)\Big\}\,ds\\
&+\int_0^t \frac{1}{N}\sum_{x\in I_N\backslash I_N^\varepsilon}g(\eta_s(x))\,\frac{1}{\varepsilon N}\sum_{z=1}^{\varepsilon N}\Delta G\left(\tfrac{x-z}{N}\right)\,ds\,.
\end{split}
\end{equation*}
Then
$\mathbb E_{\mu^N}[\mathcal R^{3}_{N,\varepsilon}(G,t)]\to 0$, as $N\to \infty$ and $\varepsilon\to 0$.
To handle the integral term in \eqref{int_2}, we use the function $\Phi:\bb R_+ \to [0, \varphi^*)$, which is the inverse function of $R$, defined in \eqref{Rphi}. We will need to introduce some more notation.
Let $\overrightarrow{\eta}^{\varepsilon N}_{s}(x)$ be the empirical density in the box of size $\varepsilon N$, which is given on, $x\in I_N^\varepsilon$, by
\begin{equation}\label{densidade_empirica}
\overrightarrow{\eta}^{\varepsilon N}_{s}(x)=\frac{1}{\varepsilon N}\sum_{y=x+1}^{x+\varepsilon N}\eta_s(y)\,.
\end{equation}
Then the  integral term in \eqref{int_2} is equal to
\begin{equation}\label{int_3}\begin{split}
&\int_0^t \frac{1}{N}\sum_{x\in I^{\varepsilon }_{N}}\Delta G\left(\tfrac{x}{N}\right)\Phi(\overrightarrow{\eta}^{\varepsilon N}_{s}(x))\,ds +\, \mathcal R^{4}_{N,\varepsilon}(G,t)\,.
\end{split}
\end{equation}
The last term above is the important expression:
\begin{equation}\label{R4}\begin{split}
& \mathcal R^{4}_{N,\varepsilon}(G,t)=\int_0^t \frac{1}{N}\sum_{x\in I^{\varepsilon }_{N}}\Delta G\left(\tfrac{x}{N}\right)\Big\{\frac{1}{\varepsilon N}\sum_{y=x+1}^{x+\varepsilon N}g(\eta_s(y))-\Phi(\overrightarrow{\eta}^{\varepsilon N}_{s}(x))\Big\}\,ds,
\end{split}
\end{equation}
which to prove that it is negligible we need to evoke the 
following very important result:
\begin{lemma}[Replacement lemma for the bulk]\label{RLemma_bulk}\footnote{The proof of this lemma is a future work.}
	For every $\delta>0$,
	\begin{equation*}
	\varlimsup_{\varepsilon\to 0}\varlimsup_{N\to\infty}\mathbb{P}_{\mu^N}\Big[\eta_\cdot\,:\,\Big|\mathcal R^{4}_{N,\varepsilon}(G,T)\Big|>\delta \,\Big]=0,
	\end{equation*}
	where $\mathcal R^{4}_{N,\varepsilon}(G,t)$ was defined in \eqref{R4}.
\end{lemma}

Note that,  $\overrightarrow{\eta}^{\varepsilon N}_{s}(x )=\langle\pi^N_{s},\iota_\varepsilon^{x/N}\rangle$, where 
$$\iota_\varepsilon^{u}(v)=\tfrac{1}{\varepsilon}{\bf{1}}_{(u,	u+\varepsilon)}(v),$$ for $u,v\in[0,1]$.
Then the integral in \eqref{int_3} can be rewritten as
\begin{equation*}\begin{split}
&\int_0^t \frac{1}{N}\sum_{x\in I^{\varepsilon }_{N}}\Delta G\left(\tfrac{x}{N}\right)\Phi\big(\langle\pi^N_{s},\iota_\varepsilon^{x/N}\rangle\big)\,ds \,.
\end{split}
\end{equation*}
Since the function inside the summation above is integrable, it is possible to prove that the last integral  is asymptotically (when $N\to\infty$ and $\varepsilon \to 0$) equal to 
\begin{equation*}\begin{split}
&\int_0^t \int_0^1\Delta G\left(u\right)\Phi\big(\langle\pi^N_{s},\iota_\varepsilon^{u}\rangle\big)\,du\,ds \,.
\end{split}
\end{equation*}

 By \eqref{assume},
we have  that, for $u\in[0,1]$  $$\langle\pi^N_{s},\iota_\varepsilon^{u}\rangle\to \int_0^1 \rho_s(v)\iota_\varepsilon^{u}(v)\,dv,$$  as $N\to\infty$.
Finally, taking $\varepsilon\to 0$, the last integral converges to $\rho_s(u)$. 
Then
the bulk term of the expression \eqref{mart_0} converges to
\begin{equation*}
\int_0^t \int_0^1\Delta G(u)\Phi(\rho_s(u))\,du\,ds = \int_0^t \langle \Phi(\rho_s) , \Delta G \rangle \,ds, 
\end{equation*}
when $N\to\infty$ and $\varepsilon \to 0$.

\bigskip

In order to analyze the boundary terms of \eqref{mart_0}, we start by observing that the expectation with respect to the probability $\mathbb P_{\mu^N}$ of
$$ \int_0^t\Big(\frac{\alpha}{N^{\theta-1}}\, G\left(0\right)-\frac{g(\eta_s(N-1))}{N^{\theta-1}}\,G\left(1\right)\Big)\,ds,$$ 
goes to zero, as  $N\to \infty$, in the case $\theta >1$, because of Lemma \ref{tecnico}.
Then, when $\theta >1$, we only need to analyze the term
\begin{align}\label{b1}
&- \int_0^t \Big(g(\eta_s(1))\partial_u G\left(0\right)-  g(\eta_s(N-1))\partial_u G\left(1\right)\Big)\,ds.
\end{align}
 In the case $\theta=1$, rewriting the boundary terms of \eqref{mart_0}, we have
\begin{align}\label{b2}
&- \int_0^t \Big(g(\eta_s(1))\partial_u G\left(0\right)-\alpha G\left(0\right)-  g(\eta_s(N-1))(\partial_u G\left(1\right)+G\left(1\right))\Big)\,ds.
\end{align}
In both cases we need to replace $g(\eta_s(1))$ and $g(\eta_s(N-1))$  by the average of $g$ in a box of size $\varepsilon N$ in a neighborhood of  $x=1$ or $x=N-1$ inside $I_N$, that is $\frac{1}{\varepsilon N}\sum_{y=2}^{1+\varepsilon N}g(\eta_s(y))$ and $\frac{1}{\varepsilon N}\sum_{y={N-1}-\varepsilon N}^{{N-2}}g(\eta_s(y))$, respectively. Note that this is similar to what we did above in \eqref{int_2}. The next step is to use the following replacement lemma with a suitable choice of $f_1$ and $f_2$.

\begin{lemma}[Replacement lemma for the boundary]\label{RLemma_boundary}\footnote{The proof of this lemma is a future work.} For  $\theta\geq 1$ and 
	for all continuous functions $f_i:[0,T]\to\mathbb R$, with $i=1,2$, and  every $\delta>0$, we have
	\begin{equation*}
	\varlimsup_{\varepsilon\to 0}\varlimsup_{N\to\infty}\mathbb{P}_{\mu^N}\Big[\eta_\cdot:\,\,\Big|\mathcal R^{b}_{N,\varepsilon}(f_1,f_2,T)\Big|>\delta \,\Big]=0,
	\end{equation*}
	where 
	\begin{equation*}
	\begin{split}
	\mathcal R^{b}_{N,\varepsilon}(f_1,f_2,T)=&\int_0^T f_1(s)\,\Big\{\frac{1}{\varepsilon N}\sum_{y=2}^{1+\varepsilon N}g(\eta_s(y))-\Phi(\overrightarrow{\eta}^{\varepsilon N}_{s}(1))\Big\}\,ds\\+&\int_0^T f_2(s)\,\Big\{	\frac{1}{\varepsilon N}\sum_{y={N-1}-\varepsilon N}^{{N-2}}g(\eta_s(y))-\Phi(\overleftarrow{\eta}^{\varepsilon N}_{s}(N\!-\!1))\Big\}\,ds\\
	\end{split}
	\end{equation*}
	where $\overrightarrow{\eta}^{\varepsilon N}_{s}(1)$ was defined in  \eqref{densidade_empirica}
	and $$\overleftarrow{\eta}^{\varepsilon N}_{s}(N-1)=\frac{1}{\varepsilon N}\sum_{y={N-1}-\varepsilon N}^{{N-2}}\eta_s(y)\,.$$
\end{lemma}

As we said above $\overrightarrow{\eta}^{\varepsilon N}_{s}(1)\sim\rho_s(0)$ and 
$\overleftarrow{\eta}^{\varepsilon N}_{s}(N\!-\!1)\sim\rho_s(1)$, then the expressions in \eqref{b1} and in \eqref{b2} converge, as $N\to\infty$ and $\varepsilon\to 0$, to
\begin{align*}
&- \int_0^t \Big(\Phi(\rho_s(0))\,\partial_u G\left(0\right)-  \Phi(\rho_s(1))\,\partial_u G\left(1\right)\Big)\,ds,
\end{align*}
in case $\theta>1$, and 
\begin{align*}
&- \int_0^t \Big(\Phi(\rho_s(0))\,\partial_u G\left(0\right)-  \Phi(\rho_s(1))\,(\partial_u G\left(1\right)+G\left(1\right))-\alpha G\left(0\right)\Big)\,ds,
\end{align*}
in case $\theta=1$. These expressions are the boundary terms in the integral equations \eqref{eq_int}, with $\kappa=0$ and $\kappa=1$, respectively.

Summarizing, the expression of the Dynkin martingale, in \eqref{int_0}, converges to the left-hand side of the integral equation \eqref{eq_int} with $\kappa=0$ and $\kappa=1$, for $\theta>1$ and $\theta=1$, respectively. Since we are only providing an idea of the proof, to make clear the notation, up to this point we have been assuming that the test $G$ does not depends on the time, that is $G\in C^2[0,1]$.

Note that from \ref{condition2.2}, $\bb P_{\mu^N}[|M_t^G|>\delta]$ vanishes as $
N\to\infty$. Recall that  $ Q^* $ is a limit point of the sequence $\{Q^N\}$, which is defined by $Q^N=\bb P_{\mu^N}(\pi^N)^{-1}$. Then, using Portmanteau Theorem, we can conclude that, in the case $\theta=1$, $Q^*$ satisfies 
\begin{equation*}
\begin{split}
Q^*\Bigg[\pi_\cdot:\,\langle \rho_{t}, G_{0}\rangle-\langle \gamma , G_{0}\rangle&  -\;  \int_{0}^{t}\big\{ \langle \rho_{s} , \partial_{s}  G_{s} \rangle + \langle \Phi(\rho_s) , \Delta G_s \rangle\big\}  \,ds \\&- \int_{0}^{t}\big\{\Phi(\rho_s(0))\partial_u G_s(0)-\Phi(\rho_s(1))\partial_u G_s(1)\big\} \,ds\\
&-\int_{0}^{t}\big\{\alpha G_s(0)-\Phi(\rho_s(1))G_s(1)\big\} \,ds \,=\,0\,,\\ &\qquad\quad\qquad\quad\forall t\in[0,T],\; \forall G\in C^{1,2}([0,T]\times[0,1]) \Bigg]= 1.
\end{split}
\end{equation*}
Note that the expression inside the probability above is the integral equation  \eqref{eq_int} with $\kappa =1$.
In the case $\theta>1$, using the same argument, we obtain a similar expression as the one above with the integral equation \eqref{eq_int} with $\kappa =0$ instead of $\kappa=1$.

\begin{remark}
	In order to show that the boundary terms $\Phi(\rho_s(0))$ and $\Phi(\rho_s(1))$ in the integral equation  \eqref{eq_int} are well defined  we need to assure that $\Phi(\rho)$ belongs to $L^2(0,T;\mathcal H^1(0,1))$. To obtain it, using Riesz representation theorem, it is enough to prove the Energy Estimate, that is: 
		\begin{equation*}
		\mathbb{E}_{ Q^*}\left[ \sup _{H}\Big\{ \int_0^T\langle \Phi(\rho_s),\partial_uH_s\rangle\,ds-c \int_0^T\langle H_s,H_s\rangle\,ds \Big\}\right] \,\leq \, M_0 < \infty,
		\end{equation*}
	for some constants $M_0$ and $c$.  As usual, the supremum above is taken over all functions $H\in C^{0,1}([0,T]\times [0,1])$ with compact support. The notation $\mathbb{E}_{Q^*}$ means the expectation with respect to the measure $ Q^*$, which is   the limit point of $Q^N$.  
\end{remark}

\section{Heuristics for hydrodynamics of the general model}\label{B}
If we had considered the more general slow boundary introduced in Remark \ref{R}, see Figure \ref{fig.2}, 
the Dynkin martingale \eqref{int_0} would be
\begin{align*}
M_t^G= & \left\langle \pi^N_t,G \right\rangle-\left\langle \pi^N_0,G \right\rangle-\int_0^t \frac{1}{N}\sum_{x=2}^{N-2}g(\eta_s(x))\Delta_N G\left(\tfrac{x}{N}\right)\,ds\nonumber \\
&- \int_0^t \Big(g(\eta_s(1))\nabla^+_N G\left(\tfrac{1}{N}\right)-  g(\eta_{s}(N-1))\nabla^-_NG\left(\tfrac{N-1}{N}\right)\Big)\,ds\\
&- \int_0^t\Bigg\{\Big(\frac{\alpha-\lambda g(\eta_{s}(1))}{N^{\theta-1}}\Big)\, G\left(\tfrac{1}{N}\right)+\Big(\frac{\beta-\delta g(\eta_{s}(N-1))}{N^{\theta-1}}\Big)\,G\left(\tfrac{N-1}{N}\right)\Bigg\}\,ds,\nonumber
\end{align*}
for  $G\in C^2[0,1]$.
Using similar ideas as in Section \ref{charac_limit_points_heu}, we get that the limit point $Q^*$ satisfies 
\begin{equation*}
\begin{split}
Q^*\Bigg[\pi_\cdot:\,\langle \rho_{t}, G_{0}\rangle-\langle \gamma , G_{0}\rangle&  -\;  \int_{0}^{t}\big\{ \langle \rho_{s} , \partial_{s}  G_{s} \rangle + \langle \Phi(\rho_s) , \Delta G_s \rangle\big\}  \,ds \\&- \int_{0}^{t}\big\{\Phi(\rho_s(0))\partial_u G_s(0)-\Phi(\rho_s(1))\partial_u G_s(1)\big\} \,ds\\
-\kappa\int_{0}^{t}\Big\{&\big(\alpha-\lambda\Phi(\rho_s(0))\big)\, G_s(0)+\big(\beta-\delta\Phi(\rho_s(1))\big)\,G_s(1)\Big\} \,ds \,=\,0\,,\\ &\qquad\quad\qquad\quad\forall t\in[0,T],\; \forall G\in C^{1,2}([0,T]\times[0,1]) \Bigg]= 1.
\end{split}
\end{equation*}
Above, $\kappa=1$ in the case $\theta=1$ and $\kappa=0$ in the case $\theta>1$. Therefore, the hydrodynamic equation is 
\begin{equation}\label{general_hid_eq}
\left\{
\begin{array}{rcll}
\partial_t \rho_t(u) &=& \Delta \Phi (\rho_t(u)),& \text{for } u\in(0,1)\text{ and } t\in(0,T],$$\\
\partial_u\Phi( \rho_t(0))&=&-\kappa\,\big(\alpha-\lambda\Phi(\rho_s(0))\big), & \text{for } t\in(0,T],$$\\
\partial_u\Phi( \rho_t(1))&=&\kappa\, \big(\beta-\delta\Phi(\rho_s(1))\big), & \text{for } t\in(0,T],$$\\
\rho_0(u)&=&\gamma(u), & \text{for } u\in[0,1].
\end{array}
\right.
\end{equation}

\begin{acknowledgement}
	A.N. was supported through a grant ``L'OR\' EAL - ABC - UNESCO Para Mulheres na Ci\^encia''. 
\end{acknowledgement}

\end{document}